\documentclass[12pt]{amsart}
\usepackage{amsfonts}
\usepackage{amsmath}
\usepackage{amssymb}
\usepackage{graphicx}
\usepackage{enumerate}
\usepackage{multicol}
\usepackage{mathrsfs}
\usepackage[usenames,dvipsnames]{pstricks}
\usepackage{epsfig}
\usepackage{pst-grad} 
\usepackage{pst-plot} 
\usepackage[hmargin=3cm,vmargin=2.5cm]{geometry}

\newtheorem{theorem}{Theorem}
\newtheorem{corollary}{Corollary}
\newtheorem{lemma}{Lemma}
\newtheorem{definition}{Definition}
\newtheorem{proposition}{Proposition}
\theoremstyle{remark}
\newtheorem{remark}{Remark}

\DeclareMathOperator{\ad}{{ad}}

\DeclareMathOperator{\Id}{Id }
\DeclareMathOperator{\spn}{span}
\DeclareMathOperator{\sign}{sign }

\DeclareMathOperator{\tr}{tr}

\title{Sub-semi-Riemannian geometry of general $H$-type groups}
\author[Mauricio Godoy M., Anna Korolko, Irina Markina]{Mauricio Godoy Molina \\ Anna Korolko \\ Irina Markina}

\address{CMAP, \'Ecole Polytechnique, CNRS, France.}
\email{mauricio.godoy@cmap.polytechnique.fr}

\address{Polytec R\&D Institute, Norway.}
\email{anna.korolko@gmail.com}

\address{Department of Mathematics, University of Bergen, Norway.}
\email{irina.markina@math.uib.no}

\thanks{The work of the first author is partially supported by the ERC Starting Grant 2009 GeCoMethods and the DIGITEO-R\'egion Ile-de-France project CONGEO. The work of the first and the third author is partially supported by NFR-FRINAT grant \#177355/V30.}



\begin{document}
\maketitle

\begin{abstract}
We introduce a special class of nilpotent Lie groups of step 2, that generalizes the so called $H$(eisenberg)-type groups, defined by A.~Kaplan in 1980. We change the presence of inner product to an arbitrary scalar product and relate the construction to the composition of quadratic forms. We present the geodesic equation for sub-semi-Riemannian metric on nilpotent Lie groups of step 2 and solve them for the case of general $H$-type groups. We also present some results on sectional curvature and the Ricci tensor of general $H$-type groups.
\end{abstract}

\section{Introduction}

A.~Kaplan in 1980 proposed the construction of Heisenberg type algebras~\cite{K1} whose commutators are intimately related to the existence of a Clifford algebra over an inner product space. He observed that the presence of a composition of two positive definite quadratic forms $\varphi$ and $\lambda$ on two vector spaces $H$ and $V$, respectively, allows one to introduce a Lie bracket $[\cdot\,,\cdot]\colon H\times H\to V$ that induces a Lie algebra structure on $H\oplus V$. He also showed that the requirement that the adjoint map on a Lie algebra $(H\oplus V,[\cdot\,,\cdot])$ of step 2 is an isometry between the orthogonal complement to its kernel in $H$ and $V$ is a necessary and sufficient condition to recover the composition of $\varphi$ and $\lambda$. The presence of a composition is related to the existence of an $H$-representation of the Clifford algebra ${\rm C}\ell(V,-\lambda)$. We emphasize that in all mentioned constructions only positive definite forms were used. Nevertheless, this restriction seems to be artificial, since compositions are defined for arbitrary bilinear non-degenerate quadratic forms. This leads to the definition of Lie brackets and, as a consequence, to the construction of general $H$(eisenberg)-type algebras that include those constructed by Kaplan as a particular case. To show that any general $H$-type algebra arises as a result of this construction is one of our main results, see Theorem 1. As in the previous case, the construction is closely related to the existence of the representation on $H$ of the Clifford algebra generated by $V$ endowed with a {\it scalar} product. 

The Heisenberg, and afterwards $H$-type, groups are core examples in the study of sub-Riemannian geometry. Remind that a sub-Riemannian manifold is a triplet $(M,\mathcal H,\rho_{\mathcal H})$, where $M$ is a smooth manifold, $\mathcal H$ is a smooth subbundle of the tangent bundle and $\rho_{\mathcal H}$ is a smoothly varying inner product defined for vectors from $\mathcal H_m$, $m\in M$. Under the bracket generating (or completely non-holonomic) condition on $\mathcal H$, the sub-Riemannian manifold can be considered as a metric space where the distance function is induced by the metric tensor $\rho_{\mathcal H}$. We give a list of references that is far from being complete, where the fundamentals of sub-Riemannian geometry can be found~\cite{ABB, BR, CDPT, M,S}.

In the same way as Riemannian and semi-Riemannian geometry are related, one can study an analogue of sub-Riemannian geometry, that we call sub-semi-Riemannian geometry. Namely, a triplet $(M,\mathcal H,\varrho_{\mathcal H})$, where $M$ and $\mathcal H$ as above but $\varrho_{\mathcal H}$ is a smoothly varying {\it scalar product} defined for vectors from $\mathcal H_m$, $m\in M$, is called a sub-semi-Riemannian manifold. The first examples and studies can be found in~\cite{CMV,G,G1,G2,GV,KM1,KM2,KM3}. Since the general $H$-type groups, introduced in the present work, carry a natural scalar product closely related to its Lie structure, we think that they will play an analogous cornerstone role in the study of sub-semi-Riemannian geometry. Apart of this mathematical interest the $H$-type groups can find applications in affine control systems, relativity theory, ADS-CFT correspondence and others subjects.

The present work is organized as following. After the introduction we review the notion of composition and its relation to Lie algebras of step 2. In Subsection~\ref{subsec_str_const} the general $H$-type algebras are defined and the main result is proved. The rest of Section~\ref{general} is dedicated to relations of the structure constants of a general $H$-type algebra and coefficients of the Clifford algebra representation. Section~\ref{sec_examples} contains examples, showing that there are general $H$-type algebras that are not among the classical ones introduced by A.~Kaplan. Moreover, we present an example showing that not all of general $H$-type algebras can be obtained by taking the classical ones and changing the natural inner product to an arbitrary scalar product. Section~\ref{h_type_group} is dedicated to the study of sub-semi-Riemannian manifold related to nilpotent  Lie groups of step 2. We write the geodesic equations and general solutions in parametric form. For the general $H$-type groups we present the closed parametric formulas for geodesics, which is possible due to the extra symmetries of the problem. The last Section~\ref{sec_curv_H_type} collects some properties about the Levi-Civita connection on general $H$-type groups, sectional curvature and the Ricci curvature tensor. 

\section{Generalized $H$-type algebras.}\label{general}

\subsection{Lie algebras of step 2}\label{Lie_step2}

Let $H$ be a Lie algebra and $V$ a vector space. A central extension of the Lie algebra $H$ by $V$ is a new Lie algebra that can be obtained as follows. Consider a bilinear skew-symmetric map $\Omega\colon H\times H\to V$ satisfying
\begin{equation}\label{jacobiOmega}
\Omega([h_1,h_2]_H,h_3)+\Omega([h_2,h_3]_H,h_1)+\Omega([h_3,h_1]_H,h_2)=0,
\end{equation}
for all $h_1,h_2,h_3\in H$ and where $[\cdot\,,\cdot]_H$ denotes the bracket in $H$. Such kind of map is called 2-cocycle. The vector space $H\oplus V$ endowed with the bracket
\[
[(h_1,v_1),(h_2,v_2)]=([h_1,h_2]_H,\Omega(h_1,h_2)),
\]
is the desired central extension.

If $H$ is abelian, then all brackets in~\eqref{jacobiOmega} vanish, and by the bilinearity of $\Omega$, equation~\eqref{jacobiOmega} holds trivially. The new brackets take the form 
\begin{equation*}
[(h_1,v_1),(h_2,v_2)]=(0,\Omega(h_1,h_2)),\quad h_1,h_2\in H,\ \ v_1,v_2\in V.
\end{equation*}
Such kind of Lie algebras with abelian $H$ will be the main object of our work and we denote them by ${\mathfrak g}=\big(H\oplus V,[\cdot\,,\cdot]\big)$. It is easy to see that $\mathfrak g$ is a nilpotent Lie algebra of step 2.
The subspace $H$ is known as the horizontal space, and $V$ is known as the vertical space. 

It is clear that the properties of the algebra $\mathfrak g$ are determined by properties of the 2-cocycle $\Omega$. One of the possible ways of constructing such~$\Omega$ was proposed by A.~Kaplan in the series of works~\cite{K1,K2,K3}, where compositions between {\it positive definite quadratic forms} were used. Such kind of algebras are called $H$-type algebras or Heisenberg-type algebras. In the present paper we propose to extend this construction by making use of {\it non-degenerate quadratic forms} admitting a composition, regardless of the sign requirement. We start from reviewing the definition and properties of composition.






\subsection{Composition of quadratic forms}\label{compositionqf}

Let $H$ and $U$ be real vector spaces, and let $\varphi\colon H\to{\mathbb R}$ and $\lambda\colon U\to{\mathbb R}$ be quadratic forms. For the rest of this article, we will assume that both $\varphi$ and $\lambda$ are non-degenerate, in the sense that the associated symmetric bilinear forms obtained by polarization
\begin{equation}\label{polarization}
\begin{array}{lcl}
\langle h_1,h_2\rangle_\varphi&=&\dfrac12\,(\varphi(h_1+h_2)-\varphi(h_1)-\varphi(h_2)),\quad h_1,h_2\in H,\\
&&\\
\langle u_1,u_2\rangle_\lambda&=&\dfrac12\,(\lambda(u_1+u_2)-\lambda(u_1)-\lambda(u_2)),\quad u_1,u_2\in U,
\end{array}
\end{equation}
satisfy the condition that if $\langle h_1,h_2\rangle_\varphi=0,$ for all $h_1\in H$, then $h_2=0$; and similarly for $\langle\cdot\,,\cdot\rangle_\lambda$. We also observe the trivial consequences of~\eqref{polarization}
\begin{equation}\label{polarization1}
\langle h,h\rangle_\varphi=\varphi(h),\qquad \langle u,u\rangle_\lambda=\lambda(u).
\end{equation} 
We emphasize that the quadratic forms $\varphi$ and $\lambda$ are not assumed to be positive definite.

\begin{definition}\label{def_composition}
A bilinear map $\mu\colon U\times H\to H$ is called a {composition} of $\varphi$ and $\lambda$ if for any $u\in U$ and any $h\in H$ the equality
\begin{equation}\label{composition}
\varphi(\mu(u,h))=\lambda(u)\varphi(h)
\end{equation}
holds. 
\end{definition}
An old problem in the theory of quadratic forms asks for conditions for the existence of a composition of two given quadratic forms. The answer to this question is a classical non-trivial application of the theory of representation of Clifford algebras, see~\cite[pp. 133--139]{Lam}.

\vspace{0.3cm}

\paragraph{\bf Example.} Let $U=H={\mathbb R}^2$ with $u=(y_1,y_2)$, $h=(x_1,x_2)$ and 
$$\varphi_a(h)=\varphi_a(x_1,x_2)=x_1^2+ax_2^2,\quad\quad\lambda_a(u)=\lambda_a(y_1,y_2)=y_1^2+ay_2^2,
$$ for any $a\in{\mathbb R}$.
The identity
\[
(y_1^2+ay_2^2)(x_1^2+ax_2^2)=(y_1x_1+ay_2x_2)^2+a(y_1x_2-y_2x_1)^2
\]
shows that the bilinear map $\mu_a\colon\mathbb R^2\times \mathbb R^2\to \mathbb R^2$ defined by
\begin{equation}\label{compR2}
\mu_a(u,h)=\mu_a\big((y_1,y_2),(x_1,x_2)\big):=(y_1x_1+ay_2x_2,y_1x_2-y_2x_1)
\end{equation}
is a composition of the quadratic forms $\varphi_a=\lambda_a$, $a\in{\mathbb R}$.

Note that for $a=0$, equation \eqref{compR2} still gives a composition of $\varphi_a$ and~$\lambda_a$. Even though the degeneracy requirement plays no role in this example, non-degeneracy is of core importance in the general arguments that will follow.


\subsection{Lie algebras and compositions}


Assume there is a composition $\mu\colon U\times H\to H$ of the quadratic forms $\varphi\colon H\to{\mathbb R}$ and $\lambda\colon U\to{\mathbb R}$. We will suppose that $\mu$ is normalized, in the sense that we choose $u_0\in U$ such that $\lambda(u_0)=\pm1$, and
\[\mu(u_0,h)=h,\quad h\in H,\]
see~\cite[p. 134]{Lam}. Let $V\subset U$ be the orthogonal complement of ${\rm span}\{u_0\}$, with respect to $\langle\cdot\,,\cdot\rangle_\lambda$, and $\pi\colon U\to V$ the corresponding orthogonal projection.
Note that the condition $\lambda(u_0)\neq 0$ is essential here, since the requirement that $u_0$ is not a null vector ($\lambda(u_0)\neq 0$) guarantees that $V=U\oplus{\rm span}\{u_0\}$, see~\cite[Lemma 2.23]{ON}.

An important fact to have in mind is that the space ${\rm End}(H)$ of endomorphisms of $H$ admits a representation of the Clifford algebra ${\rm C}\ell(V,-\lambda)$, i.e. there is an algebra homomorphism $\rho\colon {\rm C}\ell(V,-\lambda)\to {\rm End}(H)$. More precisely, such algebra homomorphism is given by $v\mapsto\mu(v,\cdot)$. To see that this is indeed the case, note that the skew-symmetry and composition formula give for any non-zero $v\in V$ and arbitrary $h'\in H$
\begin{equation}\label{repCliff}
\langle \mu\big(v,\mu(v,h)\big),h'\rangle_{\varphi}=-\langle\mu(v,h),\mu(v,h)\rangle_{\varphi} =-\langle v,v\rangle_{\lambda}
\langle h,h'\rangle_{\varphi}.
\end{equation}
We conclude that $\mu^2(v,h)\stackrel{def}{=}\mu\big(v,\mu(v,h)\big)=-\langle v,v\rangle_\lambda h$ or, in other words, $\mu^2(v,\cdot)\colon H\to H$ is the identity map multiplied by the scalar $-\lambda(v)$, which is exactly the image of the defining property of ${\rm C}\ell(V,-\lambda)$.

\begin{remark}
Up to an extra technical requirement, a converse of this result also holds, see~\cite[Chapter 5]{Lam}.
\end{remark}

In the next step we use a composition $\mu$ to define a Lie algebra structure on the vector space $H\oplus V$. To construct the corresponding Lie bracket, we first introduce a bilinear map $\Phi\colon H\times H\to U$ by means of the equality
\begin{equation}\label{defPhi}
\langle u,\Phi(h,h')\rangle_\lambda=\langle\mu(u,h),h'\rangle_\varphi,
\end{equation}
valid for all $u\in U$ and all $h,h'\in H$. The map $\Phi\colon H\times H\to U$ is in general not anti-symmetric, for an example see Subsubsection~\ref{ssRHeis}. Nevertheless, projected to $V={\rm span}\{u_0\}^{\perp}$, it has the following useful property.

\begin{proposition}\label{skew}
The map $\pi\circ\Phi:H\times H\to V$ is an anti-symmetric bilinear map, i.e. $\pi\circ\Phi(h,h')=-\pi\circ\Phi(h',h)$ for all $h,h'\in H$.
\end{proposition}

\begin{proof}
Notice that equation~\eqref{composition} can be conveniently rewritten as
\begin{equation}\label{convenient}
\langle\mu(u,h),\mu(u,h)\rangle_\varphi=\langle u,u\rangle_\lambda\langle h,h\rangle_\varphi,\quad u\in U,\ \  h\in H,
\end{equation}
by using~\eqref{polarization1}. Applying this identity to $u+u'\in U$, we obtain the equality
\begin{equation}\label{iduseful}
\langle\mu(u,h),\mu(u',h)\rangle_\varphi=\langle u,u'\rangle_\lambda\langle h,h\rangle_\varphi,
\end{equation}
by bilinearity. If in equation~\eqref{iduseful} we evaluate $u'=u_0$ and assume $u=v\in V$, then we see that
\begin{equation}\label{perp}
\langle\mu(v,h),\mu(u_0,h)\rangle_\varphi=\langle\mu(v,h),h\rangle_\varphi=0,
\end{equation}
due to the normalization imposed to $\mu$. Thus we have
\begin{multline*}
0=\langle\mu(v,h+h'),h+h'\rangle_\varphi=\\
=\underbrace{\langle\mu(v,h),h\rangle_\varphi}_{=0}+\underbrace{\langle\mu(v,h'),h'\rangle_\varphi}_{=0}+\langle\mu(v,h),h'\rangle_\varphi+\langle\mu(v,h'),h\rangle_\varphi.
\end{multline*}
 for $v\in V$. Therefore the map $\mu(v,\cdot)\colon H\to H$, $v\in V$, is a skew-symmetric map with respect to the scalar product $\langle\cdot\,,\cdot\rangle_\varphi$. This implies that for $v\in V$ and for all $h,h'\in H$ we have
\[
\langle v,\Phi(h,h')\rangle_\lambda=-\langle v,\Phi(h',h)\rangle_\lambda,
\]
which means that $\pi\circ\Phi(h,h')=-\pi\circ\Phi(h',h)$.
\end{proof}

As in Subsection~\ref{Lie_step2}, Proposition~\ref{skew} allows us to define a Lie bracket on $H\oplus V$ by
\begin{equation*}
[(h_1,v_1),(h_2,v_2)]=(0,\pi\circ\Phi(h_1,h_2)).
\end{equation*}
The resulting Lie algebra $\mathfrak{g}=(H\oplus V,[\cdot\,,\cdot])$ is a Lie algebra of step two.

The next result is completely independent of the chosen scalar products; nevertheless, proving it with the non-degenerate products induced by $\varphi$ and $\lambda$, helps us to obtain a very useful corollary.

\begin{lemma}\label{one}
Let ${\mathcal Z}(\mathfrak{g})$ denote the center of the Lie algebra $\mathfrak{g}=(H\oplus V,[\cdot\,,\cdot])$. If $g\in H\cap{\mathcal Z}(\mathfrak{g})$, then $g=0$.
\end{lemma}

\begin{proof}
If $g\in H\cap{\mathcal Z}(\mathfrak{g})$, then $g=(\tilde h,0)$, $\tilde h\in H$ and 
\[
[(\tilde h,0),(h,v)]=\big(0,\pi\circ \Phi(\tilde h,h)\big)=(0,0)\quad\text{for any}\quad(h,v)\in\mathfrak{g}.
\]
It implies that $\Phi(\tilde h,h)\in\ker\pi$ for all $h\in H$, or $\Phi(\tilde h,h)=ku_0$ for some $k\in\mathbb R$. Thus the equality
\[
\langle\mu(v,\tilde h),h\rangle_\varphi=\langle v,\Phi(\tilde h,h)\rangle_\lambda=k\langle v,u_0\rangle_\lambda=0\quad\text{for all}\quad h\in H
\]
and the non-degeneracy of the form $\varphi$ yield $\mu(v,\tilde h)=0$ for any $v\in V$. Because of~\eqref{convenient}  
\[
0=\langle\mu(v,\tilde h),\mu(v,\tilde h)\rangle_\varphi=\langle v,v\rangle_\lambda\langle \tilde h,\tilde h\rangle_\varphi,\quad\text{for any}\quad v\in V,
\]
we conclude that $\langle \tilde h,\tilde h\rangle_\varphi=0$. If the quadratic form $\varphi$ is positive definite, then we conclude that $\tilde h=0$ and finish the proof. In the case of non-degenerate indefinite form we need more careful arguments.

Let us assume that $\tilde h\neq 0$ and $\langle \tilde h,\tilde h\rangle_\varphi=0$. Using~\eqref{convenient} 
we have for arbitrary $h\in H$
\begin{align*}
\langle\mu(v,\tilde h+h),\mu(v,\tilde h+h)\rangle_{\varphi} & = \langle v,v\rangle_{\lambda}\langle \tilde h+h,\tilde h+ h\rangle_{\varphi}\\
 & = \langle v,v\rangle_{\lambda}\big(2\langle \tilde h,h\rangle_{\varphi}+\langle h, h\rangle_{\varphi}\big).
\end{align*}
On the other hand, since $\mu(v,\tilde h)=0$ for all $v\in V$, we have that
\begin{align*}
\langle\mu(v,\tilde h+h),\mu(v,\tilde h+h)\rangle_{\varphi} & = \langle\mu(v,\tilde h)+\mu(v,h),\mu(v,\tilde h)+\mu(v,h)\rangle_{\varphi} 
\\
& = \langle v,v\rangle_{\lambda}\langle h, h\rangle_{\varphi}.
\end{align*}
Comparing both sides we see that $\langle v,v\rangle_{\lambda}\langle \tilde h, h\rangle_{\varphi}=0$ for an arbitrary $v\in V$ and any $ h\in H$. This leads to a contradiction with the fact that $\tilde h\neq 0$ due to non-degeneracy of $\varphi$.
\end{proof}

As a corollary we immediately get the following. 
\begin{corollary}
If $\mu\colon V\times H\to H$ is a composition of quadratic forms $\lambda$ and $\varphi$, and for any $v\in V$ one has $\mu(v,h)=0$, then $h=0$. Similarly if $\mu(v,h)=0$ for any $h\in H$, then $v=0$.
\end{corollary}

Let us observe more properties of compositions.
\begin{itemize}
\item[1.] {Equality~\eqref{perp} shows that the map $\mu(v,\cdot)\colon H\to H$ transforms any $h\in H$ to a vector $h'=\mu(v,h)\in H$ orthogonal to $h$, for any $v\in V$, $v\neq 0$.}
\item[2.] {Formula~\eqref{iduseful} ensures that any null vector $h\in H$ ($h\neq 0$, $\langle h,h\rangle_{\varphi}=0$) is mapped to a null vector $h''=\mu(v,h)$ for any $v\in V$.}
\item[3.] {The same formula~\eqref{iduseful} implies that if $h\in H$ is fixed and $\langle h,h\rangle_{\varphi}=1$, then the map $\mu(\cdot,h)$ from $V$ to the image of $\mu(\cdot,h)$ in $H$ is an isometry and if $\langle h,h\rangle_{\varphi}=-1$ then the same map defines an anti-isometry.}
\item[4.] {For any $v\in V$ such that $\langle v,v\rangle_\lambda=1$, equation~\eqref{repCliff} shows that the map $\mu(v,\cdot)$ is an almost complex structure on $H$, that is, $\mu^2(v,\cdot)=-{\rm Id}_H$. Similarly, if $\langle v,v\rangle_\lambda=-1$, then $\mu(v,\cdot)$ is a Cartan involution on $H$, i.e. $\mu^2(v,\cdot)={\rm Id}_H$.}
\item[5.] {Lemma~\ref{one} implies that the center of the Lie algebra $\mathfrak g=(H\oplus V, [\cdot\,,\cdot])$ coincides with $V$ and the bracket $[\cdot\,,\cdot]$ defines a map $[\cdot\,,\cdot]\colon H\times H\to V$. Let us write $\ad_h(\cdot)=[h,\cdot]$, then $\ad_h$ defines a map $\ad_h\colon H\to V$. The relation~\eqref{defPhi} can be written as
\begin{equation}\label{defPhi1}
\langle v,\ad_h(h')\rangle_\lambda=\langle v,[h,h']\rangle_\lambda=\langle\mu(v,h),h'\rangle_\varphi,
\end{equation}
for $v\in V$, and $h,h'\in H$. We see that for any $h\in H$, the image of $\mu(\cdot,h)$ belongs to the orthogonal space to kernel of $\ad_h$, which we denote by $\ker^{\perp}(\ad_h)$. }
\item[6.]{Comparing the observations 3. and 5. we can assume that, for $h\neq0$, the map ${\rm ad}_h$ is an inverse map to $\mu(\cdot,h)$ and in the case of $\langle h,h\rangle_{\varphi}=\pm1$ it defines an isometry or anti-isometry from $\ker^{\perp}(\ad_h)$ to $V$.}
\item[7.] {Substituting $h'=\mu(v,h)$ with $h\neq 0$, $\langle h,h\rangle_{\varphi}=0$ in~\eqref{defPhi1}, we obtain
\[
\langle v,\ad_h(\mu(v,h))\rangle_{\lambda}=\langle\mu(v,h),\mu(v,h)\rangle_{\varphi}=\langle v,v\rangle_{\lambda}\langle h,h\rangle_{\varphi}=0
\]
for any $v\in V$. The conclusion is that $\mu(v,h)\in \ker(\ad_h)$ since the quadratic form $\lambda$ is non-degenerate.
}
\item[8.] {Taking into account 5. we conclude that the image of the map $\mu(\cdot,h)$ is $\ker(\ad_h)\cap\ker^{\perp}(\ad_h)$ for any $h\neq 0$ and $\langle h,h\rangle_{\varphi}=0$.
}
\end{itemize}

\subsection{General $H$(eisenberg)-type Lie algebras}\label{HtypeKap}


Let us, among all Lie algebras of step 2 defined in Subsection~\ref{Lie_step2} and carrying a scalar product $\langle\cdot\,,\cdot\rangle$, consider special ones that we call general $H$-type Lie algebras since they generalize the definition given in~\cite{K1} in the case when the scalar product $\langle\cdot\,,\cdot\rangle$ is an inner product (positive definite).

Let us consider a Lie algebra $\mathfrak{g}=\big(H\oplus V,[\cdot\,,\cdot]\big)$ of step two. We assume that $H\oplus V$ is endowed with a non-degenerate scalar product 
$\langle\cdot\,,\cdot\rangle=\langle\cdot\,,\cdot\rangle_H+\langle\cdot\,,\cdot\rangle_V$ such that $\langle\cdot\,,\cdot\rangle_H$ is non-degenerate scalar product on $H$ and $\langle\cdot\,,\cdot\rangle_V$ is a non-degenerate scalar product on $V$. The decomposition $H\oplus V$ becomes orthogonal with respect to $\langle\cdot\,,\cdot\rangle$. Since $V$ is the center of the Lie algebra $\mathfrak{g}$, the commutator is a map $[\cdot\,,\cdot]\colon H\times H\to V$. Let $h\in H$ be such that $\langle h,h\rangle_H\neq 0$, and we denote by $\mathfrak{H}_h$ the orthogonal complement of
\[
\ker({\rm ad}_h\colon H\to V)=\{h'\in H\colon [h,h']=0\}.
\]
We stress that $\mathfrak{H}_h$ is only defined for non-null vectors, in which case it coincides with $\ker^\bot({\rm ad}_h)$. Also note that $\mathfrak{H}_h\oplus\ker({\rm ad}_h)=H$ and $\mathfrak{H}_h\cap\ker({\rm ad}_h)=\{0\}$, see~\cite{ON}
\begin{definition}\label{general-Htype}
We say that $(\mathfrak{g},\langle\cdot\,,\cdot\rangle)$ is a Lie algebra of general $H$-type if ${\rm ad}_h:\mathfrak{H}_h\to V$ is a surjective isometry or anti-isometry for every vector $h\in H$, such that $\|h\|^2_H=\langle h,h\rangle_H=\pm1$.
\end{definition}

If $\langle\cdot\,,\cdot\rangle$ is a positive definite, then Definition~\ref{general-Htype} coincides with the definition of $H$-type groups given by A.~Kaplan~\cite{K1}.  In this context, we have the following analogue of Theorem 1 in~\cite{K1}.

\begin{theorem}
Let  $\mathfrak{g}$ be the Lie algebra constructed in Subsection~\ref{compositionqf}, by using a composition of quadratic forms $\varphi$ and $\lambda$. Then $\mathfrak{g}$ is a general $H$-type Lie algebra with $\langle\cdot\,,\cdot\rangle=\langle\cdot\,,\cdot\rangle_\varphi+\langle\cdot\,,\cdot\rangle_\lambda\vert_V$.

Conversely, for any given general $H$-type Lie algebra $\mathfrak{g}=(H\oplus V,[\cdot\,,\cdot],\langle\cdot\,,\cdot\rangle)$ there exist a vector space $U=V\oplus{\rm span}\{u_0\}$, quadratic forms $\varphi$ on $H$ and $\lambda$ on $U$ and a composition $\mu\colon U\times H\to H$ of $\varphi$ and $\lambda$, such that $\mathfrak{g}$ is built from the composition $\mu$ as in Subsection~\ref{compositionqf}.
\end{theorem}

\begin{proof}
Let $\mathfrak g$ be a step 2 Lie algebra with underlying vector space $H\oplus V$, center $V$ and composition $\mu$ of the quadratic forms $\varphi$ and $\lambda$. Then the commutator $[\cdot\,,\cdot]\colon H\times H\to V$ is defined by equation~\eqref{defPhi1}. We see that if we define the scalar product $\langle\cdot\,,\cdot\rangle$ on $\mathfrak g$ by $\langle\cdot,\cdot\rangle=\langle\cdot\,,\cdot\rangle_\varphi+\langle\cdot\,,\cdot\rangle_\lambda\vert_V$, then the map $\mu(\cdot,h)\colon V\to H$ is a formal adjoint to ${\rm ad}_h$ with respect to this scalar product. 

First, we need to prove that for each $h\in H$ with norm $\langle h,h\rangle_\varphi=\pm 1$ the map
\[
\begin{array}{ccccc}
{\rm ad}_h&\colon&\mathfrak {H}_h&\to&V\\
&&h'&\mapsto&[h,h']
\end{array}
\]
is an isometry or an anti-isometry. We start to show that it is a surjective map. Let $v'\in V$ and fix $h\in H$ such that for instance $\langle h,h\rangle_\varphi=1$. We show that $h'=\mu(v',h)$ satisfies ${\rm ad}_h(h')=v'$. According to~\eqref{defPhi1} and~\eqref{iduseful}, we have for $h'=\mu(v',h)$
\[
\langle v,{\rm ad}_h(\mu(v',h))\rangle_{\lambda}=\langle\mu(v,h),\mu(v',h)\rangle_{\varphi}=\langle v,v'\rangle_{\lambda}\langle h,h\rangle_{\varphi}=\langle v,v'\rangle_{\lambda}.
\]
Therefore 
\[
\langle v,\big({\rm ad}_h(\mu(v',h)-v')\big)\rangle_{\lambda}=0\quad\text{for any}\quad v\in V.
\]
By non-degeneracy of the quadratic form $\lambda$, we have ${\rm ad}_h(\mu(v',h))=v'$, which shows the surjectivity. If we would fix $h\in H$ with $\langle h,h\rangle_\varphi=-1$, then we need to chose $h'=\mu(v',-h)$. 

The map ${\rm ad}_h\colon \mathfrak H_h\to V$ is injective, since if $v'=0$ and $h'\in \mathfrak H_h$ then ${\rm ad}_h(h')=v'=0$
implies $h'\in \ker({\rm ad}_h)$ and we get the desired contradiction. We see that ${\rm ad}_h$ is an isomorphism between $\mathfrak H_h$ and $V$. 

To check that ${\rm ad}_h$ is an (anti-)isometry, we need to check the equality
$$
\langle{\rm ad}_h(h'),{\rm ad}_h(h'')\rangle_{\lambda}=\pm\langle h',h''\rangle_{\varphi}\quad\text{for any}\quad h',h''\in \mathfrak H_h
$$
with some fixed $h\in H$ such that $\langle h,h\rangle_{\varphi}=\pm 1$. Denote 
$$
{\rm ad}_h(h')=v'\quad\text{and}\quad{\rm ad}_h(h'')=v'',
$$
then from the previous considerations we have
$$
h'=\mu(v',h)\quad\text{and}\quad h''=\mu(v'',h).
$$
It yields
$$
\langle h',h''\rangle_{\varphi}=\langle \mu(v',h),\mu(v'',h)\rangle_{\varphi}=\langle v',v''\rangle_{\lambda}\langle h,h\rangle_{\varphi}=\pm\langle{\rm ad}_h(h'),{\rm ad}_h(h'')\rangle_{\lambda}
$$
where we used~\eqref{iduseful} in the second equality. We finish the first part of the proof.

In the other direction, the problem is more subtle. Let $\mathfrak{g}=H\oplus V$ be a general $H$-type algebra, with non-degenerate scalar product $\langle\cdot\,,\cdot\rangle=\langle\cdot\,,\cdot\rangle_H+\langle\cdot\,,\cdot\rangle_V$ and the Lie bracket $[\cdot\,,\cdot]$. At the first step we need to find a bilinear map $\mu\colon V\times H\to H$ and then extend it to the map from $U\times H$ to $H$.

We start from the following observation. The bilinear form $B\colon H\times H\to \mathbb R$ 
\[
\begin{array}{ccccc}
B_v&\colon&H\times H&\to&{\mathbb R}\\
&&(h,h')&\mapsto&\langle v,[h,h']\rangle_V
\end{array}
\]
defined for any $v\in V\setminus\{0\}$ has the following property: if $h\in H$ is fixed and $\|h\|_H^2\stackrel{def}{=}\langle h,h\rangle_H\neq 0$ then there is $\tilde h\in H$, $\tilde h\neq 0$ such that $B_v(h,\tilde h)\neq 0$. 
Indeed, fix $h\neq 0$ in $H$ with $\|h\|_H^2\neq 0$. Choose $0\neq v\in V$ and since the scalar product $\langle\cdot\,,\cdot\rangle_V$ is non-degenerate we find non-zero $v'\in V$ such that $\langle v,v' \rangle_V\neq 0$. Denote $\|h\|_H\stackrel{def}{=}\sqrt{|\langle h,h\rangle_H|}$. Since the map ${\rm ad}_{\frac{h}{\|h\|_H}}\colon \mathfrak H_{\frac{h}{\|h\|_H}}\to V$ is surjective we find the unique non-zero $h'\in  \mathfrak H_{\frac{h}{\|h\|_H}}$ such that ${\rm ad}_{\frac{h}{\|h\|_H}}(h')=v'$ and $B_v(h,h')\neq 0$ because of
$$
0\neq\langle v,v'\rangle_V=\left\langle v,{\rm ad}_{\frac{h}{\|h\|_H}}(h')\right\rangle_V=\frac{1}{\|h\|_H}B_v(h,h').
$$

Let $v\in V$ and $h\in H$. We define $\mu\colon V\times H\to H$ by the formula
\begin{equation}\label{formula1}
\langle\mu(v,h),h'\rangle_H:=\langle v,[h,h']\rangle_V.
\end{equation}
It is easy to see the following properties of $\mu$.
\begin{itemize}
\item[a)] {The map $\mu$ is bilinear.}
\item[b)] {For any $v\in V$ and $h\in H$ the element $\mu(v,h)\in \ker^{\perp}({\rm ad}_h)$.}
\item[c)] {For any non-zero $v\in V$, the map $\mu(v,\cdot)\colon H\to H$ is skew adjoint with respect to $\langle\cdot\,,\cdot\rangle_H$:
\begin{equation*}
\langle\mu(v,h),h'\rangle_H=-\langle h,\mu(v,h')\rangle_H.
\end{equation*}
}
\item[d)] {If we set $h'=h$, then the last property immediately implies that $\mu(v,h)$ is orthogonal to $h$ for arbitrary choice of $h$ and $v$.}
\item[e)] {For fixed $0\neq h\in H$ the map $\mu(\cdot,h)\colon V\to H$ is the formal adjoint to ${\rm ad}_h\colon H\to V$ with respect to the scalar product $\langle\cdot\,,\cdot\rangle$ in $\mathfrak g$.}
\end{itemize}


Now we study the properties of $\mu(\cdot,h)\colon V\to H$ for some $h\in H$ with $\langle h,h\rangle_H\neq 0$. We will show 
\begin{equation}\label{1}
[h,\mu(v,h)]=\langle h,h\rangle_H v,\quad \langle h,h\rangle_H\neq 0.
\end{equation}
and the formula
\begin{equation}\label{2}
\langle\mu(v,h),\mu(v',h)\rangle_{H}=\langle v,v'\rangle_V\langle h,h\rangle_H,\quad v,v'\in V,\  h\in H.
\end{equation}

Let $v\in V$, then since the map ${\rm ad}_h\colon \mathfrak H_h\to V$ is bijective we find the unique $\tilde h\in \mathfrak H_h$ such that ${\rm ad}_{h}(\tilde h)=v$. Then by the (anti)-isometry property we have for an arbitrary $h'\in H$
\begin{align*}
\|h\|^2\langle \tilde h,h'\rangle_H 
& 
=\pm\|h\|^2\left\langle{\rm ad}_{\frac{h}{\|h\|}} \tilde h,{\rm ad}_{\frac{h}{\|h\|}} h'\right\rangle_V=\pm\langle[h,\tilde h],[h,h']\rangle_V
\\
&
=\pm\langle v,[h,h']\rangle_V=\pm\langle\mu(v,h),h'\rangle_H.
\end{align*}
Thus
$ \|h\|^2 \tilde h=\mp\mu(v,h).
$
Then
$$
[h,\mu(v, h)]=\pm[ h,\|h\|^2 \tilde h]=\pm\|h\|^2 v.
$$

Now we move to show the composition formula~\eqref{2}. Let $v,v'\in V$ and $h\in H$ with $ \|h\|^2\neq 0$. Then
\begin{equation*}
\langle\mu(v,h),\mu(v',h)\rangle_H \stackrel{\eqref{formula1}}{=}\langle v,[h,\mu(v',h)]\rangle_V\stackrel{\eqref{1}}{=}\pm\|h\|^2\langle v,v'\rangle_V.
\end{equation*}

To show equalities~\eqref{1} and~\eqref{2} for $h\neq0$ with $\|h\|_H^2=0$ we use the continuity properties of linear maps. To proceed, we choose an orthonormal basis, see~\cite[p. 50]{ON}, and consider coordinates with respect to this basis. It can be easily seen that the following arguments do not depend on this choice. Let $h_n$ be a sequence of non-null vectors in $H$ such that $h_n\to h$ as $n\to \infty$ coordinate-wise. Then $\mu(v,h_n)\to \mu(v,h)$ as $n\to\infty$ coordinate-wise in $H$ for any fixed $v\in V$. Since the Lie bracket and scalar product are continuous maps we conclude that~\eqref{1} implies that for null vector $h\in H$ the image $\mu(v,h)$ belongs to $\ker\ad_h$ for any $v\in V$. The equality~\eqref{2} shows that the image $\mu(v,h)$ of a null vector $h\in H$ is a null vector. 

Having the equality ~\eqref{2} for all $ h\in H$, we substitute $h$ with an arbitrary sum $h+h'$ and obtain
$$
\langle\mu(v,h),\mu(v',h')\rangle_H+\langle\mu(v',h),\mu(v,h')\rangle_H=2\langle v,v'\rangle_V\langle h,h'\rangle_H.
$$
Applying skew-symmetry we get
$$
-\langle\mu(v,\mu(v',h')),h\rangle_H-\langle\mu(v',\mu(v,h')),h\rangle_H=2\langle v,v'\rangle_V\langle h',h\rangle_H
$$
or in other words
$$
\mu(v,\mu(v',\cdot))+\mu(v',\mu(v,\cdot))=-2\langle v,v'\rangle_V\Id_H(\cdot),\quad v,v\in V,
$$
where $\Id_H$ is the identity map in $H$. Particularly, for $v=v'$ we deduce
$$
\mu^2(v,\cdot)=\mu(v,\mu(v,\cdot))=-\langle v,v\rangle_V\Id_H(\cdot)
$$
Thus 
$$
\langle\mu(v,h),\mu(v,h)\rangle_H=-\langle\mu(v,\mu(v,h)),h\rangle_H=\langle v,v\rangle_V\langle h,h\rangle_H,
$$
and we showed the composition formula~\eqref{convenient} for $\mu\colon V\times H\to H$. With this we recover all the properties of the bilinear map $\mu\colon V\times H\to H$, listed in items 1.-8.

The next step is to extend the map $\mu$ to a vector space $U$. Set $U=V\oplus\mathbb R$ with the scalar product $\langle\cdot\,,\cdot\rangle_U=\langle\cdot\,,\cdot\rangle_V+\langle\cdot\,,\cdot\rangle_{\mathbb R}$, where $\langle\cdot\,,\cdot\rangle_{\mathbb R}$ is usual  Euclidean product. Define an extended bilinear map $\tilde \mu\colon U\times H\to H$ by
$$
\tilde\mu(v+\alpha,h):=\mu(v,h)+\alpha h,\quad v\in V, \ \ h\in H,\ \ \alpha\in \mathbb R.
$$
Then
\begin{align*}
\langle\tilde\mu(v+\alpha,h),\tilde\mu(v+\alpha,h)\rangle_U
&=
\langle\mu(v,h),\mu(v,h)\rangle_V+\alpha^2\langle h,h\rangle_H
\\
&+
2\alpha\langle h,\mu(v,h)\rangle.
\end{align*}
The last term in the right hand side vanishes due to the property c). Applying the composition formula for $\mu$ we obtain
$$
\langle\tilde\mu(v+\alpha,h),\tilde\mu(v+\alpha,h)\rangle_U=\langle h,h\rangle_H\big(\langle v,v\rangle_V+\langle \alpha,\alpha\rangle_\mathbb R\big),
$$
that shows the composition $\tilde\mu\colon U\times H\to H$ of the quadratic forms $\lambda(\cdot)=\langle\cdot\,,\cdot\rangle_U$ and $\varphi(\cdot)=\langle\cdot\,,\cdot\rangle_H$. 

Remark, that we can use a negatively definite product $\langle\cdot\,,\cdot\rangle_{\mathbb R}$ in $\mathbb R$, but in this case we define the scalar product on $U$ by
$\langle\cdot\,,\cdot\rangle_U=\langle\cdot\,,\cdot\rangle_V-\langle\cdot\,,\cdot\rangle_{\mathbb R}$ to get the same result.
This finishes the proof.
\end{proof}

\subsection{Structure constants of a general $H$-type algebra}\label{subsec_str_const}


Consider the scalar product vector spaces $(H,\langle\cdot\,,\cdot\rangle_H)$ and $(V,\langle\cdot\,,\cdot\rangle_V)$ of indices $\nu_H$ and $\nu_V$ respectively. Recall that the index of a scalar product vector space is the dimension of a maximal subspace where the scalar product is negative definite. Let from now on $(h_1,\ldots, h_n)$ and $(v_1,\ldots, v_m)$ be orthonormal bases of the vector spaces $H$ and $V$, respectively. We assume for the rest of the paper that they are ordered in such a way that they satisfy $\langle h_i,h_j\rangle_H=\varepsilon_i^{\nu_H}\delta_{ij}$ and $\langle v_{\alpha},v_{\beta}\rangle_V=\varepsilon_{\alpha}^{\nu_V}\delta_{\alpha\beta}$, where $\delta_{lk}$ is the Kronecker symbol and $\varepsilon_k^\nu$ is the sign symbol defined by
\[
\varepsilon_k^{\nu}=
\begin{cases}
-1\quad & \text{if}\quad k\leq \nu
\\
1\quad & \text{otherwise}.
\end{cases}
\]
In addition, denote by $J_H=(\langle h_i,h_j\rangle_H)=(\varepsilon_i^{\nu_H}\delta_{ij})$ and $J_V=(\langle v_\alpha,v_\beta\rangle_V)=(\varepsilon_{\alpha}^{\nu_V}\delta_{\alpha\beta})$ the Gram matrices of $(H,\langle\cdot\,,\cdot\rangle_H)$ and $(V,\langle\cdot\,,\cdot\rangle_V)$ with respect to the chosen bases.

Let $\mathfrak g=(H\oplus V,[\cdot\,,\cdot],\langle\cdot\,,\cdot\rangle=\langle\cdot\,,\cdot\rangle_H+\langle\cdot\,,\cdot\rangle_V)$ be a general $H$-type algebra. As we saw in Section~\ref{HtypeKap}, the Lie bracket defines an endomorphism $\mu(v,\cdot)\colon H\to H$ for any non vanishing $v\in V$ by formula~\eqref{formula1}.  We write
\begin{equation}\label{Clifford_coef}
\mu(v_{\alpha},h_i)=\sum_{j=1}^{n}A^{\alpha}_{ij}h_j,
\end{equation}
and
\begin{equation}\label{structural_const}
[h_i,h_j]=\sum_{\beta=1}^{m}B^{\beta}_{ij}v_{\beta}.
\end{equation}
Since the Lie algebra $\mathfrak{g}$ has step 2, the numbers $B^{\beta}_{ij}$ are exactly the structure constants of $\mathfrak{g}$.

\begin{proposition}\label{relation_cliff_struc}
The coefficients $A^{\alpha}_{ij}$ and the structure constants $B^{\beta}_{ij}$ of the general $H$-type Lie algebra are related by
\[
\varepsilon_{j}^{\nu_H}A^{\alpha}_{ij}=\varepsilon_{\alpha}^{\nu_V}B^{\alpha}_{ij}
\]
\end{proposition}

\begin{proof}
We calculate
\begin{align}\label{one_side}
\langle \mu(v_{\alpha},h_i),h_j\rangle_H
&
=\langle v_{\alpha},[h_i,h_j]\rangle_V=\sum_{\beta=1}^{m}B^{\beta}_{ij}\langle v_{\alpha},v_{\beta}\rangle_V
\\
&
=\sum_{\beta=1}^{m}B^{\beta}_{ij}\varepsilon_{\alpha}^{\nu_V}\delta_{\alpha,\beta}=\varepsilon_{\alpha}^{\nu_V}B^{\alpha}_{ij}.\nonumber
\end{align}
From the other side 
\begin{equation}\label{other_side}
\langle \mu(v_{\alpha}h_i),h_j\rangle_H
=\sum_{k=1}^{n}A^{\alpha}_{ik}\langle h_{k},h_{j}\rangle_V
=\sum_{k=1}^{n}A^{\alpha}_{ik}\varepsilon_{k}^{\nu_H}\delta_{kj}=\varepsilon_{j}^{\nu_H}A^{\alpha}_{ij}.
\end{equation}
Comparing~\eqref{one_side} and~\eqref{other_side} we obtain the result.
\end{proof}


\section{Examples}\label{sec_examples}


\subsection{General Heisenberg groups}\label{sLHeis}


As one of our main motivating examples, we consider the case of the Heisenberg Lie algebra. Consider the real nilpotent Lie algebra $\mathfrak{heis}^{2n+1}$ with generators
\[
X_1,\dotsc,X_n,Y_1,\dotsc,Y_n,Z
\]
satisfying the well-known commutator rules
\[
[X_k,Y_l]=\delta_{k,l}Z,\quad [X_k,Z]=[Y_l,Z]=0,
\]
where $k,l\in\{1,\dotsc,n\}$ and $\delta_{k,l}$ denotes the Kronecker symbol. The horizontal subspace is defined by
\[
H={\rm span}\{X_1,\dotsc,X_n,Y_1,\dotsc,Y_n\}\subset\mathfrak{heis}^{2n+1}
\]
and the vertical space is simply $V={\rm span}\{Z\}$. We denote by ${\mathscr H}^{2n,\nu_H,1}$ the Lie algebra $\mathfrak{heis}^{2n+1}$, endowed with the non-degenerate scalar product on $H$ of index $\nu_H$. The super-index $(2n,\nu_H,1)$ refers to the dimension of $H$, its index and the dimension of the center of $\mathfrak{heis}^{2n+1}$. For example, the classical sub-Riemannian structure on $\mathfrak{heis}^{2n+1}$ is denoted by ${\mathscr H}^{2n,0,1}$.

The simplest non-trivial example of the objects studied in Section~\ref{general} is the general Heisenberg group ${\mathscr H}^{2,1,1}$ endowed with a sub-Lorentzian metric~\cite{G,KM1}. To simplify the notation, consider the generators $X,Y,Z$ satisfying
\[
[X,Y]=Z,\quad [X,Z]=[Y,Z]=0.
\]
As above, let $H={\rm span}\{X,Y\}\subset\mathfrak{heis}^3$ and define a non-degenerate bilinear form on it by
\[
\langle X,X\rangle_H=-\langle Y,Y\rangle_H=1,\quad\langle X,Y\rangle_H=0,
\]
and then extending it linearly to all of $H$. Additionally, consider the inner product $\langle Z,Z\rangle_V=1$ extended linearly to all of $V={\rm span}\{Z\}$.


\subsubsection{${\mathscr H}^{2,1,1}$ as a general $H$-type group}\label{ssRHeis}


To see that ${\mathscr H}^{2,1,1}$ indeed satisfies the conditions as in Subsection~\ref{HtypeKap}, note that if $h=\alpha X+\beta Y\in H$, for $\alpha,\beta\in{\mathbb R}$, then
\[
\langle h,h\rangle_H=\langle\alpha X+\beta Y,\alpha X+\beta Y\rangle_H=\alpha^2-\beta^2.
\]
Assume $\alpha^2-\beta^2=1$, thus $\alpha\neq0$. The case in which $\alpha^2-\beta^2=-1$ can be treated analogously. Note that
\begin{align*}
[\alpha X+\beta Y,\alpha X+\beta Y]&=\alpha\beta Z-\alpha\beta Z=0,\\
\langle\alpha X+\beta Y,\beta X+\alpha Y\rangle_H&=\alpha\beta-\alpha\beta=0,
\end{align*}
therefore
\[
\ker{\rm ad}_h={\rm span}\,\{\alpha X+\beta Y\}\quad\mbox{and}\quad\mathfrak{H}_h={\rm span}\,\{\beta X+\alpha Y\}.
\]
An immediate consequence of the above calculations is that
\[
{\rm ad}_h(\beta X+\alpha Y)=[\alpha X+\beta Y,\beta X+\alpha Y]=(\alpha^2-\beta^2)Z=Z
\]
is an isometry between the vector spaces $\mathfrak{H}_h$ and $V$.


\subsubsection{${\mathscr H}^{2,1,1}$ in terms of composition of quadratic forms}\label{quadHeis}


In this sense, the above construction has a very clear interpretation, given by formula~\eqref{compR2} for $a=-1$. Proceeding as in Subsection~\ref{compositionqf}, we have that if $H=U={\mathbb R}^2$, and $\varphi=\lambda$ are such that $\varphi(h)=\varphi(x_1,x_2)=x_1^2-x_2^2$ and $\lambda(h)=\varphi(y_1,y_2)=y_1^2-y_2^2$, then 
\[
\varphi(x_1,x_2)\lambda(y_1,y_2)=\varphi(\tilde\mu((y_1,y_2),(x_1,x_2)),
\]
where \(\tilde\mu((y_1,y_2),(x_1,x_2))=(y_1x_1-y_2x_2,y_1x_2-y_2x_1)\).
Let us fix $u_0=(1,0)\in U$, which satisfies $\lambda(u_0)=1$. Then
\[
\tilde\mu(u_0,h)=\tilde\mu\big((1,0),(x_1,x_2)\big)=(x_1,x_2),
\]
 for $h=(x_1,x_2)\in H$, thus $\tilde\mu(u_0,\cdot)$ is simply the identity map of $H$. The orthogonal complement of ${\rm span}\,\{u_0\}$ is $V={\rm span}\,\{v\}$, where $v=(0,1)$. Fix $\pi\colon U\to V$ to be the orthogonal projection. It is easy to check that the map $\Phi\colon H\times H\to U$ determined by equation~\eqref{defPhi} is given by
\[
\Phi((x_1,x_2),(x_1',x_2'))=(x_1x_1'-x_2x_2',x_2x_1'-x_1x_2'),
\]
and therefore the space $H\oplus V$ inherits the Lie algebra structure given by
\[
[((x_1,x_2),v_1),(x_1',x_2'),v_2)]=(0,(x_2x_1'-x_1x_2')v),
\]
for any $h=(x_1,x_2)\in H$, $h'=(x_1',x_2')\in H$ and any $v_1,v_2\in V$. Taking a basis of $H\oplus V$ given by $X=((1,0),0)$, $Y=((0,1),0)$ and $Z=((0,0),v)$, we see that
\[
[X,Y]=Z,\quad[X,Z]=[Y,Z]=0.
\]


\subsubsection{Care needs to be taken}\label{nonexistence}


Picking incompatible quadratic forms on $H$ and $U$ can have undesirable consequences. For example, consider the problem of finding a composition of the quadratic forms
\[
\tilde\varphi(x_1,x_2)=x_1^2-x_2^2\quad\mbox{and}\quad\tilde\lambda(y_1,y_2)=y_1^2+y_2^2,
\]
defined on $H=U={\mathbb R}^2$, respectively. In order to solve the problem, we need to determine the coefficients $a,b,c,d,\alpha,\beta,\gamma,\delta\in{\mathbb R}$ for the bilinear form
\begin{multline*}
\tilde\mu((y_1,y_2),(x_1,x_2))=\\(ay_1x_1+by_1x_2+cy_2x_1+dy_2x_2,\alpha y_1x_1+\beta y_1x_2+\gamma y_2x_1+\delta y_2x_2)
\end{multline*}
to satisfy the composition rule
\[
\tilde\varphi(x_1,x_2)\tilde\lambda(y_1,y_2)=\tilde\varphi(\tilde\mu((y_1,y_2),(x_1,x_2))).
\]
If such a map exists, then the following equations must hold
\begin{align}
a^2-\alpha^2=b^2-\beta^2=-c^2+\gamma^2=-d^2+\delta^2&=1\label{coeffsq}\\
ab-\alpha\beta=ac-\alpha\gamma=ad-\alpha\delta=bc-\beta\gamma&=\nonumber\\
=bd-\beta\delta=cd-\gamma\delta&=0.\label{coeffmix}
\end{align}
From equation~\eqref{coeffsq}, it follows that $a,b,\gamma,\delta\neq0$. This in turn implies $\alpha,\beta,c,d\neq0$, by using the first and last equation in~\eqref{coeffmix}. Thus none of the coefficients of $\tilde\mu$ can vanish. Since $ab=\alpha\beta$, the first equation in~\eqref{coeffsq} can be rewritten as
\[
\frac{\alpha^2\beta^2}{b^2}-\alpha^2=\alpha^2\left(\frac{\beta^2-b^2}{b^2}\right)=-\frac{\alpha^2}{b^2}=1.
\]
But then $\alpha^2+b^2=0$, which gives the desired contradiction.


\subsubsection{Extensions to higher dimensions}\label{higher_dim}


It is interesting to observe that a sort of ``product'' construction holds for the general Heisenberg groups ${\mathscr H}^{2n,n,1}$. Consider non-degenerate bilinear form given by
\[
\langle X_k,X_l\rangle_H=-\langle Y_k,Y_l\rangle_H=\delta_{k,l},\quad\langle X_k,Y_l\rangle_H=0,
\]
for $k,l\in\{1,\dotsc,n\}$, and then extending it linearly to all of $H$. The inner product $\langle Z,Z\rangle_V=1$ is extended linearly to all of $V={\rm span}\{Z\}$. It is a simple exercise to note that the arguments in Subsubsection~\ref{ssRHeis} can be easily generalized to this case.

Similarly to what was done in Subsubsection~\ref{quadHeis}, we consider the vector spaces $H={\mathbb R}^{2n}$ and $U={\mathbb R}^2$ with quadratic forms
\begin{align*}
\varphi(h)=\varphi(x_1,y_1,\dotsc,x_n,y_n)&=x_1^2-y_1^2+\dotsb+x_n^2-y_n^2,\\
\lambda(u)=\lambda(u_1,u_2)&=u_1^2-u_2^2,
\end{align*}
on $H$ and $U$ respectively. A composition $\mu$ between these two quadratic forms is simply the bilinear map $\tilde\mu\colon U\times H\to H$ given by
\begin{multline*}
((u_1,u_2),(x_1,y_1,\dotsc,x_n,y_n))\mapsto\\
(u_1x_1-u_2y_1,u_1y_1-u_2x_1,\dotsc,u_1x_n-u_2y_n,u_1y_n-u_2x_n),
\end{multline*}
which corresponds to a coordinate-wise version of the composition law in Subsubsection~\ref{quadHeis}.

Note that the discussion in Subsubsection~\ref{nonexistence} can be extended to the higher dimensional case. Before drawing the conclusion, we recall the law of inertia for real quadratic forms, which says that up to a change of coordinates every quadratic form of  index $\nu$ acting on ${\mathbb R}^N$ can be written as $-x_1^2-\cdots-x_\nu^2+x_{\nu+1}^2+\cdots+x_N^2$. Combining this and Subsubsection~\ref{nonexistence}, we see that any other choice of scalar product on $H\subset\mathfrak{heis}^{2n+1}$ with index different from $n$ will not give a general $H$-type Lie algebra.


\subsection{General Quaternionic $H$-type groups}


We can try to proceed as in Subsection~\ref{sLHeis} for a quaternionic analogue to the Heisenberg group ${\mathscr H}^{2,1,1}$. Of course, such case is more delicate than the previous one, nevertheless some of the constructions follow the same patterns. To start with these ideas, note that the following identity holds
\begin{multline*}
(x_1^2+a\,x_2^2+b\,x_3^2+ab\,x_4^2)(y_1^2+a\,y_2^2+b\,y_3^2+ab\,y_4^2)=\\
(x_1y_1+a\,x_2y_2+b\,x_3y_3+ab\,x_4y_4)^2+a(-x_1y_2+x_2y_1-b\,x_3y_4+b\,x_4y_3)^2+\\
b(-x_1y_3+x_3y_1+a\,x_2y_4+a\,x_4y_2)^2+ab(-x_1y_4+x_4y_1-x_2y_3+x_3y_2)^2,
\end{multline*}
for any $a,b\in{\mathbb R}$. This implies that
\begin{multline*}
\mu((y_1,y_2,y_3,y_4),(x_1,x_2,x_3,x_4))=\\
(y_1x_1+a\,y_2x_2+b\,y_3x_3+ab\,y_4x_4,y_1x_2-y_2x_1-b\,y_3x_4+b\,y_4x_3,\\
y_1x_3-y_3x_1+a\,y_2x_4+a\,y_4x_2,y_1x_4-y_4x_1-y_2x_3+y_3x_2)
\end{multline*}
is a composition of the quadratic form 
\begin{equation}\label{quadquat}
\varphi_{(a,b,ab)}(x_1,x_2,x_3,x_4)=x_1^2+a\,x_2^2+b\,x_3^2+ab\,x_4^2
\end{equation}
on $H={\mathbb R}^4$ with $\lambda_{(a,b,ab)}=\varphi_{(a,b,ab)}$ on $U=\mathbb R^4$. Consider the vector spaces $H=U={\mathbb R}^4$, each endowed with the quadratic form equal to~\eqref{quadquat}. By the inertia theorem, we can reduce our study to the following two situations:
\[
1.\; a=b=1,\qquad 2.\,a=-b=1.
\]
In the first case, we obtain the quaternionic $H$-type group ${\mathscr H}^{4,0,3}$ studied in~\cite{CM}. The second case presents a new example of a general $H$-type group, that we denote by ${\mathcal H}^{4,2,3}$, following the notation in Subsubsection~\ref{ssRHeis}.

Let us fix $u_0=(1,0,0,0)\in U$, which satisfies $\lambda_{(1,-1,-1)}(u_0)=1$. Then for all $h=(x_1,x_2,x_3,x_4)\in H$ we have
\[
\mu((1,0,0,0),(x_1,x_2,x_3,x_4))=(x_1,x_2,x_3,x_4)
\]
The orthogonal complement of ${\rm span}\,\{u_0\}$ is $V={\rm span}\,\{v_1,v_2,v_3\}$, where $v_1=(0,1,0,0)$, $v_2=(0,0,1,0)$ and $v_3=(0,0,0,1)$. Fix $\pi\colon U\to V$ to be the orthogonal projection. It is easy to check that the map $\Phi\colon H\times H\to U$ determined by equation~\eqref{defPhi} is given by
\begin{multline*}
\Phi((x_1,x_2,x_3,x_4),(x_1',x_2',x_3',x_4'))=\\
(x_1x_1'+x_2x_2'-x_3x_3'-x_4x_4',x_2x_1'-x_1x_2'+x_4x_3'-x_3x_4',\\
x_3x_1'+x_4x_2'-x_1x_3'-x_2x_4',x_4x_1'-x_3x_2'+x_2x_3'-x_1x_4')
\end{multline*}
and therefore the space $H\oplus V$ has a Lie algebra structure given by
\begin{multline*}
[((x_1,x_2,x_3,x_4),w_1),(x_1',x_2',x_3',x_4'),w_2)]=\\
(0,(x_2x_1'-x_1x_2'+x_4x_3'-x_3x_4')v_1+(x_3x_1'+x_4x_2'-x_1x_3'-x_2x_4')v_2+\\
+(x_4x_1'-x_3x_2'+x_2x_3'-x_1x_4')v_3),
\end{multline*}
for any $(x_1,x_2,x_3,x_4),(x_1',x_2',x_3',x_4')\in H$ and any $w_1,w_2\in V$. Taking a basis of $H\oplus V$ given by 
\[
X_1=((1,0,0,0),0),\quad X_2=((0,1,0,0),0)
\]
\[
X_3=((0,0,1,0),0),\quad X_4=((0,0,0,1),0)
\]
and $Z_i=(0,v_i)$, we see that
\begin{equation*}
\begin{array}{ccccr}
\mbox{$[X_1,X_2]$}&=&[X_3,X_4]&=&-Z_1,\\
\mbox{$[X_2,X_3]$}&=&[X_1,X_4]&=&Z_2,\\
\mbox{$[X_1,X_3]$}&=&[X_4,X_2]&=&Z_3,
\end{array}
\end{equation*}
which is isomorphic as Lie algebras to the one presented in~\cite{CM}. Observe that this construction can be extended as in Subsubsection~\ref{higher_dim} to obtain the general $H$-type groups ${\mathscr H}^{4n,2n,3}$. 


\section{General $H$-type Lie groups}\label{h_type_group}



\subsection{Step 2 Lie groups}


Let $\mathbb G$ be the unique (up to isomorphism) connected simply connected Lie group with Lie algebra $\mathfrak g$. It is known, see~\cite{FS}, that for nilpotent groups, we can identify $\mathbb G$ and $\mathfrak g$ via the globally diffeomorphic exponential map $\exp\colon \mathfrak g\to \mathbb G$. Let $(h_1,\ldots, h_n)$ and $(v_1,\ldots, v_m)$ be bases of the vector spaces $H$ and $V$, respectively, with structure constants $C^\alpha_{ij}$ defined by the formula $[h_i,h_j]=\sum_{\alpha=1}^{m}C^{\alpha}_{ij}v_{\alpha}$. We use the normal coordinates
\begin{multline*}
\mathbb G\ni g=(x_1,\ldots,x_n,t_1,\ldots t_m) \buildrel{\exp}\over\longleftrightarrow\\
\left(\sum_{i=1}^{n}x_ih_i+\sum_{\alpha=1}^{m}t_{\alpha}v_{\alpha}\right)=(h,v)\in H\oplus V.
\end{multline*}

By the Baker-Campbell-Hausdorff formula, the product law in $\mathbb G$ can be written as
\begin{equation}\label{defproduct}
g\cdot g'=(x,t)\cdot(x',t')=\left(x+x',t+t'+\frac{1}{2}[h,h']\right).
\end{equation}
The group $(\mathbb G,\cdot)$ is referred to as a nilpotent Lie group of step two, see also~\cite{M}. 
We calculate the coordinates of the Lie bracket $[h,h']$. Let $h=\sum_{i=1}^nx_ih_i$ and $h'=\sum_{j=1}^nx'_jh_j$, then
\begin{align}
[h,h']=\sum_{i=1}^n\sum_{j=1}^nx_ix'_j[h_i,h_j]=\sum_{i,j=1}^nx_ix'_j\sum_{\alpha=1}^mC^{\alpha}_{ij}v_{\alpha}
\end{align}
and the group law~\eqref{defproduct} is written as
\begin{equation}\label{defproduct1}
g\cdot g'=(x,t)\cdot(x',t')=\left(\sum_{i=1}^n(x_i+x'_i),\sum_{\alpha=1}^m \left(t_{\alpha}+t'_{\alpha}+\frac{1}{2}\sum_{i,j=1}^nx_ix'_jC^{\alpha}_{ij}\right)\right).
\end{equation}


\subsection{Step 2 Lie groups as sub-semi-Riemannian manifolds}\label{ssR_step2}


For each $g\in \mathbb G$, we denote by $L_g\colon \mathbb G\to \mathbb G$ the left-translation diffeomorphism. If $e\in \mathbb G$ denotes the identity element, as usual we identify the tangent space $T_e\mathbb G$ with the Lie algebra ${\mathfrak g}$. The left-translations, allow us to define the distributions
\begin{equation*}
\mathcal H(g)=d_eL_g(H),\quad \mathcal V(g)=d_eL_g(V).
\end{equation*}
We also define a left invariant indefinite metric $\rho$ by translating the scalar product $\langle\cdot\,,\cdot\rangle$ from $\mathfrak g$ to an arbitrary point:
\[
\rho(X_g,X'_g)=\langle d_gL_{g^{-1}}(X_g),d_gL_{g^{-1}}(X'_g)\rangle,\quad g\in \mathbb G, X_g,X'_g\in T_g\mathbb G.
\]
Let us assume that $H$ and $V$ are orthogonal with respect to $\langle\cdot\,,\cdot\rangle$, and let us write $\langle\cdot\,,\cdot\rangle=\langle\cdot\,,\cdot\rangle_H+\langle\cdot\,,\cdot\rangle_V$. The metric $\rho_{\mathcal H}$ obtained by translating $\langle\cdot\,,\cdot\rangle_H$ is the indefinite metric such that at each $g\in \mathbb G$ we have $\rho_H\colon\mathcal H(g)\times\mathcal H(g)\to\mathbb R$ and it is called a {\it sub-semi-Riemannian metric.} The triplet $\big(\mathbb G, \mathcal H, \rho_{\mathcal H}\big)$ is an example of a {\it sub-semi-Riemannian manifold}, see~\cite{KM3}. By the well-known theorem of Chow and Rashevski{\u\i}, see~\cite{C,R}, since the space $H$ Lie-generates the whole Lie algebra ${\mathfrak g}$, every two points on $\mathbb G$ can be connected by a piecewise smooth curve with velocity vectors almost everywhere in $\mathcal H$.

Differentiating~\eqref{defproduct1} with respect to $g'$ we get the matrix of $d_eL_g$. Applying it to the basic vectors $\{\partial_{x_i}\}_{i=1}^{n}$ and 
$\{\partial_{t_{\alpha}}\}_{\alpha=1}^{m}$ at the identity $e=(0,\ldots,0)\in \mathbb G$ we obtain the expressions of left invariant vector fields at point $g=(x_1,\ldots,x_n,t_{1},\ldots,t_m)\in \mathbb G$:
\begin{equation}\label{vector_field}
X_i(g)=\partial_{x_i}+\frac{1}{2}\sum_{\alpha=1}^m\sum_{j=1}^nx_jC^{\alpha}_{ji}\partial_{t_{\alpha}}=\partial_{x_i}+\frac{1}{2}\sum_{\alpha=1}^m\sum_{j=1}^nC^{\alpha}_{ij}x_j\partial_{t_{\alpha}},
\end{equation}
for $i=1,\ldots,n$, and \(T_{\alpha}(g)=\partial_{t_{\alpha}}\), \(\alpha=1,\ldots,m\).

%
%

\begin{remark}
Among all sub-semi-Riemannian manifolds related to the Lie algebras of step 2 the sub-semi-Riemannian manifolds produced by general $H$-type groups occupy a special place due to additional natural relations between scalar product and Lie brackets expressed in Proposition~\ref{relation_cliff_struc}.
\end{remark}

Our aim is to calculate the parametric formulas for normal sub-semi-Riemannian geodesics that are projections of solutions of a Hamiltonian system. We start from the deduction of such system, but we cannot solve it in the case of an arbitrary step 2 Lie group. Nevertheless, we are able to give precise formulas for the case of general $H$-type Lie groups. 

For the subsequent computations, we choose orthonormal bases $\{h_i\}_{i=1}^{n}$ of $H$ and $\{v_{\beta}\}_{\alpha=1}^{m}$ of $V$ as in Subsection~\ref{subsec_str_const}. In such case, the vector fields given by~\eqref{vector_field} are orthonormal with respect to $\rho_{\mathcal H}$.
The metric Hamiltonian function is given by $\mathbf H(g,\lambda)=\frac{1}{2}\rho_{\mathcal H}^{*}(\lambda_g,\lambda_g)$, where $\rho_{\mathcal H}^{*}\colon T_g^*\mathbb G\times T_g^*\mathbb G\to \mathbb R$ is the co-metric, and the co-vector $\lambda_g\in T_g^*\mathbb G$ is expressed in coordinates of the dual basis as 
$$
\lambda(g)=\sum_{i=1}^{n}\xi_i\,dx_i+\sum_{\alpha=1}^{m}\theta_{\alpha}dt_{\alpha}.
$$
Details about the construction of co-metric can be found in~\cite{G,KM3,M,S}. By making use of the orthonormal left-invariant vector fields~\eqref{vector_field}, we write the Hamiltonian in the form 
\begin{equation}\label{ham_1}
\mathbf H(g,\lambda)=-\frac{1}{2}\sum_{i=1}^{\nu_H}\lambda(X_i(g))^2+\frac{1}{2}\sum_{i=\nu_H+1}^{n}\lambda(X_i(g))^2.
\end{equation}
Since
\[
\lambda(X_i(g))=\sum_{i=1}^{n}\xi_i+\frac{1}{2}\sum_{\alpha=1}^m\sum_{j=1}^nC^{\alpha}_{ij}x_j\theta_{\alpha},\quad i=1,\dotsc,n,
\]
we need to simplify the notations in equation~\eqref{ham_1}. We write $\Omega_{ij}=\sum_{\alpha=1}^mC^{\alpha}_{ij}\theta_{\alpha}$, $\xi=(\xi_1,\ldots,\xi_n)$, $\theta=(\theta_1,\ldots,\theta_m)$, $x=(x_1,\ldots,x_n)$, $t=(t_1,\ldots,t_m)$, and finally, 
\[
\Omega x=\big((\Omega x)_1,\ldots,(\Omega x)_n\big)=\left(\sum_{j=1}^n\Omega_{1j}x_j,\ldots,\sum_{j=1}^n\Omega_{nj}x_j\right).
\]
Then the Hamiltonian function is given by
\[
\mathbf H(x,t,\xi,\theta)=\frac{1}{2}\langle \xi,\xi\rangle_H+\frac{1}{2}\langle\xi,\Omega x\rangle_H+\frac{1}{8}\langle\Omega x,\Omega x\rangle_H,
\]
and the Hamiltonian system is
\begin{equation}\label{Ham_system}
\begin{cases}
\dot x=\frac{\partial H}{\partial \xi}=J_H\xi+\frac{1}{2}J_H\Omega x,
\\
\dot t_{\alpha}=\frac{\partial H}{\partial \theta_{\alpha}}=\frac{1}{2}\langle \xi,C^{\alpha}x\rangle_H+\frac{1}{4}\langle \Omega x,C^{\alpha}x\rangle_H,\quad\alpha=1,\ldots,m,
\\
\dot \xi=-\frac{\partial H}{\partial x}=\frac{1}{2}\Omega J_H\xi+\frac{1}{4}\Omega J_H\Omega x,
\\
\dot \theta_{\alpha}=-\frac{\partial H}{\partial t_{\alpha}}=0,
\end{cases}
\end{equation}
where $C^\alpha=(C^\alpha_{ij})$. Since we are interested in the projection of the solution to the Hamiltonian system into $\mathbb G$, we rewrite the system in terms of the relevant variables $x$ and $t$. The first and the third equations imply
\begin{eqnarray*}
\frac{1}{2}J_H\Omega\dot x &=\frac{1}{2}J_H\Omega J_H\xi+\frac{1}{4}(J_H\Omega)^2x,
\\
J_H\dot\xi &=\frac{1}{2}J_H\Omega J_H\xi+\frac{1}{4}(J_H\Omega)^2x,
\end{eqnarray*}
and we conclude that $J_H\dot\xi=\frac{1}{2}J_H\Omega\dot x$. Now differentiating the first equation of~\eqref{Ham_system} and substituting the last equality we obtain
\[
\ddot x=J_H\Omega \dot x.
\]
Observing that $J_H\dot x=\xi+\frac{1}{2}\Omega x$, we calculate
\[
\dot t_{\alpha}=\frac{1}{2}\langle \xi+\frac{1}{2}\Omega x,C^{\alpha}x\rangle_H=\frac{1}{2}\langle J_H\dot x,C^{\alpha}x\rangle_H=\frac{1}{2}\dot x^T C^{\alpha}x,
\]
where $^T$ denotes the usual transposition of matrices. We conclude that the geodesic equations are 
\begin{equation}\label{geodesic_eq}
\begin{cases}
\ddot x& =J_H\Omega \dot x,
\\
\dot t_{\alpha}&=\frac{1}{2}\dot x^T C^{\alpha}x,\quad\alpha=1,\ldots,m.
\end{cases}
\end{equation}
It is easy to see that we can give a closed formula for the solutions of the first of the above equations, provided $\Omega$ is invertible. Working in a Lie group, without loss of generality we can assume that $x(0)=0$, and let us denote $\dot x(0)=V_0$. These solutions are given by
\begin{equation*}
x(s)=V_0(J_H\Omega)^{-1}(\exp(J_H\Omega s)-I).
\end{equation*}
The solutions for the second equation can be found by integration, but no closed formula seems to be available. Note that if $\theta_{\alpha}=0$ for all $\alpha$, then $\ddot x=0$ and thus $x(s)=V_0s$. It follows that 
\[
\dot t_\alpha=\frac12\dot x^T C^\alpha x=\frac{s}2V_0^T C^\alpha V_0=0,
\]
since $C^\alpha$ is a skew-symmetric matrix, and thus all the $t_{\alpha}$'s are constant. We conclude that the geodesics are straight lines in the space $t=t(0)$, passing through the point $(0,t(0))$ with velocity vector $V_0$.

\begin{remark}
The equation for horizontal coordinates depends on the choice of scalar product, but the equation for the vertical components does not. This last equation expresses the horizontality condition, which is independent from the chosen scalar product and depend only on the structure constants of the Lie algebra. For a positive definite metric for $H$-type groups the equations~\eqref{geodesic_eq} where solved in~\cite{K2,K3}, see also~\cite{CCM,CM}. The case of the Lorentzian metric and a metric of index 2 for some specific Lie groups of step 2 can be found in~\cite{G,KM1,KM2}.
\end{remark}


\subsection{Geodesics for general $H$-type groups}


Here we present solutions of~\eqref{geodesic_eq} with initial data $x(0)=0$, $t(0)=0$, $\dot x(0)=V_0$, $\theta_{\alpha}(0)=\theta_{\alpha}$, $\alpha=1,\ldots,m$, in the case of general $H$-type groups.

Since the case of $\theta=0$ was already considered, let us assume that not all of the $\theta_{\alpha}$'s vanish. Denoting $\dot x=y$ in~\eqref{geodesic_eq}, we obtain
\[
\dot y(s)=J_H\Omega\, y(s)\quad\Longrightarrow\quad y(s)=V_0\exp(sJ_H\Omega).
\]
To present the structure of $\exp(sJ_H\Omega)$,
we need to present some useful formulas that will simplify the calculations afterward. Following our conventions, we denote the constants associated to $\mu$ by $A_{ij}^\alpha$, as in equation~\eqref{Clifford_coef}, and the structure constants of the general $H$-type algebra $\mathfrak{g}$ by $B_{ij}^\alpha$, as in equation~\eqref{structural_const}. First of all, note that the equation $\varepsilon_{j}^{\nu_H}A^{\alpha}_{ij}=\varepsilon_{\alpha}^{\nu_V}B^{\alpha}_{ij}$ of Proposition~\ref{relation_cliff_struc} can be written as
\begin{equation}\label{AB}
\varepsilon_{\alpha}^{\nu_V}A^{\alpha}=B^{\alpha}J_H.
\end{equation}
In addition, we also need the following lemma.

\begin{lemma}
With the notations introduced above and choosing bases as in Subsection~\ref{subsec_str_const}, we have that
\begin{eqnarray}
\big(A^{\alpha}\big)^2=-\|v_{\alpha}\|^2J_H=-\varepsilon_{\alpha}^{\nu_V}J_H.\label{A2}\\
\big(J_HB^{\alpha}\big)^2=-\varepsilon_{\alpha}^{\nu_V}J_H.\label{JB2}\\
J_HB^{\alpha}J_HB^{\beta}+J_HB^{\beta}J_HB^{\alpha}=0.\label{Bab}
\end{eqnarray}
Denoting $\Theta^2=\sum_{\alpha=1}^{m}\varepsilon_{\alpha}^{\nu_V}\theta_{\alpha}^2$ for the square norm of the initial co-vector, we get
\begin{equation}\label{JHO2}
\big(J_H\Omega\big)^2=-\Theta^2 J_H.
\end{equation}
\end{lemma}

\begin{proof}
To see that equation~\eqref{A2} holds, we calculate from one side
\begin{align*}
\langle\mu(v_{\alpha},(\mu(v_{\alpha},h_i)),h_j\rangle_H
&
=\langle\mu^2(v_{\alpha},h_i),h_j\rangle_H=-\langle v_{\alpha},v_{\alpha}\rangle_V\langle h_i,h_j\rangle_H
\\
&=-\|v_{\alpha}^2\|_V\varepsilon _i^{\nu_h}\delta_{ij},
\end{align*}
and from the other side
\begin{equation*}
\langle\mu(v_{\alpha},(\mu(v_{\alpha},h_i)),h_j\rangle_H=\sum_{k=1}^{n}A^{\alpha}_{ik}\langle\mu(v_{\alpha},h_k),h_j\rangle_H=
\sum_{k=1}^{n}A^{\alpha}_{ik}A^{\alpha}_{kj}=\big\{\big(A^{\alpha}\big)^2\big\}_{ij}.
\end{equation*}

Formula~\eqref{JB2} follows directly from~\eqref{AB} and~\eqref{A2}. Explicitly
\begin{equation*}
\big(J_HB^{\alpha}\big)^2 =J_H\big(B^{\alpha}J_H\big)^2J_H=J_H\big(\varepsilon_{\alpha}^{\nu_V}A^{\alpha}\big)^2J_H=-\varepsilon_{\alpha}^{\nu_V}J_H.
\end{equation*}

To prove equation~\eqref{Bab}, remind the relation
$$
\langle\mu(v_{\alpha},\mu(v_{\beta},h_i)),h_j\rangle_H+\langle\mu(v_{\beta},\mu(v_{\alpha},h_i)),h_j\rangle_H=-2\langle v_{\alpha},v_{\beta}\rangle_V\langle h_i,h_j\rangle_H.
$$
It implies $\big(A^{\alpha}A^{\beta}+A^{\beta}A^{\alpha}\big)_{ij}=-2\langle v_{\alpha},v_{\beta}\rangle_V\varepsilon_i^{\nu_H}\delta_{ij}=0$ with respect to the orthonormal basis $(v_1,\ldots,v_m)$. We immediately deduce that
\begin{equation*}
J_HB^{\alpha}J_HB^{\beta}+J_HB^{\beta}J_HB^{\alpha}=\varepsilon_{\alpha}^{\nu_V}\varepsilon_{\beta}^{\nu_V}J_H\big(A^{\alpha}A^{\beta}+A^{\beta}A^{\alpha}\big)J_H=0.
\end{equation*}

Finally, to obtain equation~\eqref{JHO2}, we have the following chain of equalities, which are simple applications of formulas~\eqref{JB2} and~\eqref{Bab}
\begin{align*}
\big(J_H\Omega\big)^2
& =\left(\sum_{\alpha=1}^{m}\theta_{\alpha}J_HB^{\alpha}\right)^2=\sum_{\alpha=1}^{m}\theta_{\alpha}^2\big(J_HB^{\alpha}\big)^2
\\
&+
\sum_{\alpha,\beta=1,\,\alpha\neq\beta}^{m}\theta_{\alpha}\theta_{\beta}\big(J_HB^{\alpha}J_HB^{\beta}+J_HB^{\beta}J_HB^{\alpha}\big)
\\
&=-\sum_{\alpha=1}^{m}\varepsilon_{\alpha}^{\nu_V}\theta_{\alpha}^2J_H:=-\Theta^2 J_H.\qedhere
\end{align*}
\end{proof}

The value $\Theta^2<0$ corresponds the timelike initial data, $\Theta^2>0$ corresponds the spacelike initial data and $\Theta^2=0$ corresponds the null initial co-vector.

If $\Theta^2=0$ then $\exp(sJ_H\Omega)=I+sJ_H\Omega$. The solution is given by
\[
x(s)=\left(sI+\frac{s^2}{2}J_H\Omega\right)V_0.
\]
To solve the equation $ \dot t_{\alpha}=\frac{1}{2}\dot x^TB^{\alpha}x$
we notice that $J_HB^{\alpha}=B^{\alpha}J_H$ by~\eqref{JB2}, that implies also $J_H\Omega=\Omega J_H$ for $\Omega=\sum_{\beta=1}^{m}\theta_{\beta}B^{\beta}$. Remind also that $B^{\alpha}$, $\alpha=1,\ldots,m$ are skew-symmetric by definition. Thus, we get
\begin{align*}
 \dot t_{\alpha}(s)& =\frac{s}{2}\Big[V_0^T\big(I+sJ_H\Omega\big)^TB^{\alpha}\big(I+\frac{s}{2}J_H\Omega\big)V_0\Big]
 \\
& =\frac{s}{2}\Big[V_0^TB^{\alpha}V_0+\frac{s^2}{2}\big(\Omega V_0\big)^TB^{\alpha}\big(\Omega V_0\big)-\frac{s}{2}\langle V_0,B^{\alpha}\Omega V_0\rangle_H\Big].
\end{align*}
The integration over $[0,s]$ leads to the formula
$$
t_{\alpha}(s)=\frac{s^2}{4}V_0^TB^{\alpha}V_0-\frac{s^3}{12}\langle V_0,B^{\alpha}\Omega V_0\rangle_H+\frac{s^4}{16}\big(\Omega V_0\big)^TB^{\alpha}\big(\Omega V_0\big).
$$

Let us assume that $\Theta^2\neq 0$.
Then the exponential $\exp(sJ_H\Omega)$ can be decomposed into four series
\begin{align*}
\exp(sJ_H\Omega)&=
I\Big(\sum_{n=0}^{\infty}\frac{(s\Theta)^{4n}}{(4n)!}\Big)+J_H\Omega s \Big(\sum_{n=0}^{\infty}\frac{(s\Theta)^{4n}}{(4n+1)!}\Big)
\\
&-\sign(\Theta^2)J_H\Big(\sum_{n=0}^{\infty}\frac{(s\Theta)^{4n+2}}{(4n+2)!}\Big)-\sign(\Theta^2)\Omega s \Big(\sum_{n=0}^{\infty}\frac{(s\Theta)^{4n+2}}{(4n+3)!}\Big)
\\
&
=\frac{1}{2}\Big[I\Big(\cos(s\Theta)+\cosh(s\Theta)\Big)+J_H\Omega\frac{\sin(s\Theta)+\sinh(s\Theta)}{\Theta}
\\
&
-\sign(\Theta^2)J_H\Big(-\cos(s\Theta)+\cosh(s\Theta)\Big)
\\
&-\sign(\Theta^2)\Omega\frac{-\sin(s\Theta)+\sinh(s\Theta)}{\Theta}\Big].
\end{align*}
where we write $\Theta=\sqrt{|\Theta^2|}$.

Solving the equation $\dot x(s)=V_0\exp(sJ_H\Omega)$ on the interval $[0,s]$, we find
\begin{align}\label{last_solution}
x(s)&
=\frac{1}{2}\Big[I\frac{\sin(s\Theta)+\sinh(s\Theta)}{\Theta}+J_H\Omega\frac{-\cos(s\Theta)+\cosh(s\Theta)}{|\Theta|^2}\nonumber
\\
&
-\sign(\Theta^2)J_H\frac{-\sin(s\Theta)+\sinh(s\Theta)}{\Theta}
\\
&-\sign(\Theta^2)\Omega\frac{\cos(s\Theta)+\cosh(s\Theta)-2}{|\Theta|^2}\Big].\nonumber
\end{align}

We present only the expression for $\dot t_{\alpha}$, since the integration of the presented formulas are tedious but simple and we will not work with the formulas of geodesics anymore.
\begin{tiny}
\begin{align}\label{geod_tdot}
\dot t_{\alpha}(s) &= \frac{B^{\alpha}}{\Theta}\big(\cos(s\Theta)+\cosh(s\Theta)\big)\big(\sin(s\Theta)+\sinh(s\Theta)\big)\nonumber
\\
&+\frac{J_HB^{\alpha}\Omega}{|\Theta^2|}\big(\cosh^2(s\Theta)-\cos^2(s\Theta)\big)\nonumber
\\
&-\frac{\sign(\Theta^2)J_HB^{\alpha}}{\Theta}\big(\cos(s\Theta)+\cosh(s\Theta)\big)\big(-\sin(s\Theta)+\sinh(s\Theta)\big)\nonumber
\\
&-\frac{\sign(\Theta^2)B^{\alpha}\Omega}{|\Theta^2|}\big(\cos(s\Theta)+\cosh(s\Theta)\big)\big(\cos(s\Theta)+\cosh(s\Theta)-2\big)\nonumber
\\
& -\frac{J_H\Omega B^{\alpha}}{|\Theta^2|}\big(\sin(s\Theta)+\sinh(s\Theta)\big)^2\nonumber
\\
& -\Big(\frac{\Omega B^{\alpha}\Omega}{|\Theta^3|}+\frac{\sign(\Theta^2)J_HB_{\alpha}}{\Theta}\Big)\big(\sin(s\Theta)+\sinh(s\Theta)\big)\big(-\cos(s\Theta)+\cosh(s\Theta)\big)\nonumber
\\
&+2\frac{\sign(\Theta^2)\Omega B^{\alpha}}{|\Theta^2|}\big(\sinh^2(s\Theta)-\sin^2(s\Theta)\big)
\\
&+\frac{\sign(\Theta^2)J_H\Omega B^{\alpha}\Omega}{|\Theta^2|}\big(\sin(s\Theta)+\sinh(s\Theta)\big)\big(\cos(s\Theta)+\cosh(s\Theta)-2\big)\nonumber
\\
&-\frac{\sign(\Theta^2) B^{\alpha}\Omega}{|\Theta^2|}\big(-\cos(s\Theta)+\cosh(s\Theta)\big)^2\nonumber
\\
&+\Big(\frac{B^{\alpha}}{\Theta}+\frac{\sign(\Theta^2)J_H\Omega B_{\alpha}\Omega}{|\Theta^3|}\Big)\big(-\cos(s\Theta)+\cosh(s\Theta)\big)\big(-\sin(s\Theta)+\sinh(s\Theta)\big) \nonumber
\\
&+\frac{J_H B^{\alpha}\Omega}{|\Theta^2|}\big(-\cos(s\Theta)+\cosh(s\Theta)\big)\big(-\cos(s\Theta)+\cosh(s\Theta)-2\big)\nonumber
\\
& -\frac{J_H\Omega B^{\alpha}}{|\Theta^2|}\big(-\sin(s\Theta)+\sinh(s\Theta)\big)^2\nonumber
\\
&
-\frac{\Omega B^{\alpha}\Omega}{|\Theta^3|}\big(-\sin(s\Theta)+\sinh(s\Theta)\big)\big(-\cos(s\Theta)+\cosh(s\Theta)-2\big).\nonumber
\end{align}
\end{tiny}

We summarize the result of the section in the following theorem
\begin{theorem}
Let $\mathbb G$ be a general $H$-type group. Then the geodesics starting at the origin with initial horizontal velocity $V_0=(V_0^1,\ldots,V_0^n)\in H$ and initial co-vector $(\theta_1,\ldots,\theta_m)$ are given by:
\begin{itemize}
\item
{$
x(s)=V_0s,\quad t_{\alpha}=0,\ \ \alpha=1,\ldots,m\quad\text{if}\quad \theta_{1}=\ldots=\theta_m=0.
$
}
\item 
{ If $\Theta^2=\sum_{\alpha=1}^{m}\varepsilon_{\alpha}^{\nu_V}\theta_{\alpha}^2=0$ then
\begin{align*}
x(s)&=\left(sI+\frac{s^2}{2}J_H\Omega\right)V_0,
\\
t_{\alpha}(s)& =\frac{s^2}{4}V_0^TB^{\alpha}V_0-\frac{s^3}{12}\langle V_0,B^{\alpha}\Omega V_0\rangle_H+\frac{s^4}{16}\big(\Omega V_0\big)^TB^{\alpha}\big(\Omega V_0\big),
\end{align*}
for $\alpha=1,\ldots,m$,
}
\item
{In the case of $\Theta^2\neq 0$ the coordinates $x$ are given by~\eqref{last_solution}, and the coordinates $t$ are obtained by integrating formula~\eqref{geod_tdot}.
}
\end{itemize}
\end{theorem}

Notice that a further knowledge about the matrices $B^{\alpha}$, $\alpha=1,\ldots,m$ could lead to simplifications in the formula~\eqref{geod_tdot}. For example in the case of $H$-type groups with positive definite metric under the presence of the so-called $J^2$-condition~\cite{CDKR} the formulas for geodesics were described in~\cite{CCM}. In the presence of the same $J^2$-condition the geodesics for the Heisenberg group ${\mathscr H}^{2,1,1}$ and for ${\mathscr H}^{4,2,3}$, can be found in~\cite{KM1,KM2}.


\section{Curvatures for general $H$-type groups}\label{sec_curv_H_type}


\subsection{Covariant derivatives of left-invariant vector fields}


Let ${\mathbb G}$ be a general $H$-type Lie group with a left-invariant metric $\rho=\rho_{\mathcal H}+\rho_{\mathcal V}$ as in Subsection~\ref{ssR_step2}, with Lie algebra $\mathfrak{g}$. For notational convenience, since we work with left-invariant vector fields, we use $\langle\cdot\,,\cdot\rangle=\langle\cdot\,,\cdot\rangle_H+\langle\cdot\,,\cdot\rangle_V$ instead. Let $\nabla$ be the semi-Riemannian Levi-Civita connection on ${\mathbb G}$, as in~\cite{ON}. Milnor~\cite{Milnor} observed that the Koszul formula for left-invariant vector fields $X,Y,Z\in\mathfrak{g}$ reduces to
\begin{equation}\label{koszul}
\langle\nabla_XY,Z\rangle=\frac12\Big(\langle X,[Z,Y]\rangle-\langle Y,[X,Z]\rangle-\langle Z,[Y,X]\rangle\Big).
\end{equation}
Milnor made this observation in the Riemannian case, but his argument goes unchanged for the semi-Riemannian case. In particular, equation~\eqref{koszul} implies the skew-symmetry of the operator $\nabla_X$, that is $\langle\nabla_XY,Z\rangle+\langle Y,\nabla_XZ\rangle=0$.

Equation~\eqref{koszul} helps computing covariant derivatives of left-invariant vector fields in the case of general $H$-type groups, whether they are horizontal or vertical. We summarize these formulas in the following lemma.

\begin{lemma}\label{covderlemma}
Let $h,h_1,h_2\in H$ and $v,v_1,v_2\in V$, then we have
\begin{equation}\label{covder}
\begin{cases}
\nabla_{h_1}h_2=\frac12[h_1,h_2],\\
\nabla_vh=\nabla_hv=-\frac12\mu(v,h),\\
\nabla_{v_1}v_2=0.
\end{cases}
\end{equation}
\end{lemma}

\begin{proof}
All the computations follow easily from equations~\eqref{formula1} and~\eqref{koszul}. To prove the first identity, we see that
\begin{equation*}
\langle\nabla_{h_1}h_2,h\rangle=\frac12\Big(\langle h_1,\underbrace{[h,h_2]}_{\mbox{\footnotesize{vertical}}}\rangle-\langle h_2,\underbrace{[h_1,h]}_{\mbox{\footnotesize{vertical}}}\rangle-\langle h,\underbrace{[h_2,h_1]}_{\mbox{\footnotesize{vertical}}}\rangle\Big)=0,
\end{equation*}
thus $\nabla_{h_1}h_2$ does not have any horizontal component. For the vertical direction we have
\begin{equation*}
\langle\nabla_{h_1}h_2,v\rangle=\frac12\Big(\langle h_1,\underbrace{[v,h_2]}_{=0}\rangle-\langle h_2,\underbrace{[h_1,v]}_{=0}\rangle-\langle v,[h_2,h_1]\rangle\Big)=-\frac12\langle[h_2,h_1],v\rangle.
\end{equation*}
For the first part of the second equation, we have
\begin{align*}
\langle\nabla_{v}h,h_1\rangle&=\frac12\Big(\langle v,[h_1,h]\rangle-\langle h,\underbrace{[v,h_1]}_{=0}\rangle-\langle h_1,\underbrace{[h,v]}_{=0}\rangle\Big)\\&=-\frac12\langle v,[h,h_1]\rangle=-\frac12\langle\mu(v,h),h_1\rangle
\end{align*}
in the horizontal direction, and
\begin{equation*}
\langle\nabla_{v}h,v_1\rangle=\frac12\Big(\langle v,\underbrace{[v_1,h]}_{=0}\rangle-\langle h,\underbrace{[v,v_1]}_{=0}\rangle-\langle v_1,\underbrace{[h,v]}_{=0}\rangle\Big)=0
\end{equation*}
 in the vertical direction. All other cases are similar or easier.
\end{proof}


\subsection{Sectional curvatures and stable subspaces}\label{subsec_sec_curv}


We want to compute the semi-Riemannian curvature endomorphism for the general $H$-type groups, using the formulas in Lemma~\ref{covderlemma}. We use the sign conventions of Lee~\cite{Lee} for the curvature endomorphism, that is
\[
R(X,Y)Z=\Big(\nabla_X\nabla_Y-\nabla_Y\nabla_X-\nabla_{[X,Y]}\Big)Z,
\]
for any $X,Y,Z\in{\rm Vect}({\mathbb G})$.

\begin{lemma}\label{curv_end_lemma}
Let $h,h_1,h_2\in H$ and $v,v_1,v_2\in V$, then we have
\begin{equation}\label{curv_end}
\begin{cases}
R(h_1,h_2)h=\frac14\Big(2\mu([h_1,h_2],h)-\mu([h,h_1],h_2)-\mu([h_2,h],h_1)\Big),\\
R(h_1,v)h_2=-\frac14[h_1,\mu(v,h_2)],\\
R(h,v_1)v_2=-\frac14\mu(v_1,\mu(v_2,h)),\\
R(v_1,v_2)v=0.
\end{cases}
\end{equation}
\end{lemma}

\begin{proof}
All these formulas follow easily from~\ref{covderlemma}. Other combinations follow from skew-symmetry in the first two components or from the first Bianchi identity.
\end{proof}

Our aim is to determine certain non-degenerate subspaces of $H\subset{\mathfrak g}$ that are stable under the action of the composition. 
\begin{definition}
A non-degenerate subspace $W\subset H$ is stable if $\mu(v,W)\subset W$ for some non-null vector $v\in V$ or, in other words, if $W$ is an invariant subspace for the map $\mu(v,\cdot)\colon H\to H$.
\end{definition}

We focus our attention in the case of two dimensional subspaces, since some of their properties are intimately related to the corresponding sectional curvatures. Now let us compute the different sectional curvatures $k(P)$ that are possible for different choices of planes $P\subset{\mathfrak g}$.

\begin{proposition}\label{sec_curv_lemma}
Let $h,h_1,h_2\in H$ and $v,v_1,v_2\in V$. Assume, furthermore, that the pairs $h_1,h_2$ and $v_1,v_2$ are linearly independent. Then
\begin{equation}\label{sec_curv}
k(P)=\begin{cases}
\dfrac34\,\dfrac{\|[h_1,h_2]\|^2}{\|h_1\|^2\|h_2\|^2-\langle h_1,h_2\rangle^2},&\mbox{if }P=\spn\{h_1,h_2\},\\ &\\
-\dfrac14&\mbox{if }P=\spn\{h,v\},\\ &\\
0,&\mbox{if }P=\spn\{v_1,v_2\}.
\end{cases}
\end{equation}
\end{proposition}

\begin{proof}
Recall that the sectional curvature of the plane $P=\spn\{\alpha,\beta\}$ is given by
\begin{equation}\label{def_sec_curv}
k(P)=\frac{\langle R(\alpha,\beta)\alpha,\beta\rangle}{\|\alpha\|^2\|\beta\|^2-\langle\alpha,\beta\rangle^2}.
\end{equation}
To compute the numerator in~\eqref{def_sec_curv}, we use the formulas~\eqref{curv_end} in the different cases listed. If $P=\spn\{h_1,h_2\}\subset H$ we have
\begin{align*}
\langle R(h_1,h_2)h_1,h_2\rangle&=\frac12\langle\mu([h_1,h_2],h_1),h_2\rangle-\frac14\langle\mu(\underbrace{[h_1,h_1]}_{=0},h_2),h_2\rangle-\\
&-\frac14\langle\mu([h_2,h_1],h_1),h_2\rangle=\frac34\langle\mu([h_1,h_2],h_1),h_2\rangle\\
&=\frac34\langle[h_1,h_2],[h_1,h_2]\rangle=\frac34\|[h_1,h_2]\|^2.
\end{align*}
The first formula follows. In the case when $P=\spn\{h,v\}$, we know that $\langle h,v\rangle=0$ by definition. What is left to prove is a straightforward computation
\begin{align*}
\langle R(h,v)h,v\rangle&=-\frac14\langle[h,\mu(v,h)],v\rangle=-\frac14\langle\mu(v,h),\mu(v,h)\rangle\\
&=-\frac14\langle v,v\rangle \langle h,h\rangle=-\frac14\|v\|^2\|h\|^2.
\end{align*}
The second formula follows. Finally, the last equality holds since $R(v_1,v_2)v=0$ for all $v,v_1,v_2\in V$, and thus $k(\spn\{v_1,v_2\})=0$.
\end{proof}

We conclude with the computation of the sectional curvature of stable and abelian two dimensional horizontal subspaces of ${\mathfrak g}$.

\begin{proposition}\label{st_ab_planes}
Let $P=\spn\{h_1,h_2\}$ be a non-degenerate two dimensional horizontal subspace of $\mathfrak{g}$. If $P$ is stable then $k(P)=\frac34$, and if $P$ is abelian then $k(P)=0$.
\end{proposition}

\begin{proof}
Without loss of generality, we can assume that $\{h_1,h_2\}$ is a non-null orthonormal basis, since the sectional curvature is invariant under choice of basis, and $P$ is non-degenerate. Let us first write $h_2=h+\mu(v,h_1)$ for some non-null $v\in V$ and $h\in H$ chosen such that $h\in\mu(V,h_1)^\bot$. For any $w\in V$ we have the following chain of equalities
\begin{align*}
\langle w,[h_1,h_2]\rangle_V&=\langle w,[h_1,h]\rangle_V+\langle w,[h_1,\mu(v,h_1)]\rangle_V\\
&=\underbrace{\langle\mu(w,h_1),h\rangle_H}_{=0}+\langle\mu(w,h_1),\mu(v,h_1)\rangle_H=\|h_1\|^2\langle w,v\rangle_V,
\end{align*}
where we used equality~\eqref{repCliff}. It follows that $[h_1,h_2]=\|h_1\|^2 v$. It is clear that $h=0$ if and only if $P$ is stable. In this case, note that
\begin{equation*}
\|[h_1,h_2]\|^2=\langle v,v\rangle=\|h_1\|^2\langle\mu(v,h_1),\mu(v,h_1)\rangle={\|h_2\|^2}{\|h_1\|^2},
\end{equation*}
therefore $k(P)=\frac34\,\frac{\|[h_1,h_2]\|^2}{\|h_1\|^2\|h_2\|^2-\langle h_1,h_2\rangle^2}=\frac34$.

On the other hand, if $P$ is abelian, then $[h_1,h_2]=0$. The result follows.
\end{proof}

\begin{remark}
The calculations in Proposition~\ref{st_ab_planes} show that
\[
\|[h_1,h_2]\|^2=\|h_1\|^2\|h_2-h\|^2,
\]
thus the sectional curvature $k(P)$ takes the form
\begin{align*}
k(P)&=\frac34\,\frac{\|h_2-h\|^2}{\|h_2\|^2}=\frac34\,\frac{\|h_2\|^2-2\langle h+\mu(v,h_1),h\rangle+\|h\|^2}{\|h_2\|^2}\\
&=\frac34\,\frac{\|h_2\|^2-\|h\|^2}{\|h_2\|^2}=\left(1-\frac{\|h\|^2}{\|h_2\|^2}\right).
\end{align*}
This is enough to prove that $0\leq k(P)\leq\frac34$ in~\cite{K2}, but this no longer holds in the semi-Riemannian case. Some bounds are still available, but do not give any new information.
\end{remark}


\subsection{The Ricci tensor and scalar curvature}


With the results of Subsection~\ref{subsec_sec_curv} at hand, we can compute the Ricci tensor for general $H$-type groups and obtain some interesting properties of the scalar curvature. 

\begin{theorem}\label{ricci_theorem}
Let $(h_1,\ldots, h_n)$ and $(v_1,\ldots, v_m)$ be non-degenerate orthonormal bases of $H$ and $V$, respectively, ordered as in Subsection~\ref{subsec_str_const}. Then the Ricci tensor for the general $H$-type group ${\mathbb G}$, relative to the chosen basis, has the diagonal form
\begin{equation}\label{ricci_tensor}
\left(
\begin{array}{c|c}
\displaystyle{-\frac12\sum_{\alpha=1}^m\varepsilon_\alpha^{\nu_V}}\,J_H&\mbox{\large{$0$}}\\
\hline
\mbox{\large{$0$}}&\displaystyle{\frac14\sum_{i=1}^n\varepsilon_i^{\nu_H}}\,J_V
\end{array}
\right).
\end{equation}
\end{theorem}

\begin{proof}
The Ricci tensor ${\rm Ric}(X,Y)$ evaluated at $X,Y\in{\rm Vect}({\mathbb G})$ is given by
\[
{\rm Ric}(X,Y)=\sum_{i=1}^n\langle R(h_i,X)Y,h_i\rangle+\sum_{\alpha=1}^m\langle R(v_\alpha,X)Y,v_\alpha\rangle.
\]
It is clear we need to consider only three cases, depending whether $X$ or $Y$ are horizontal or vertical. For most of the computations, we use the formulas in Lemma~\ref{curv_end}.

Let $h,\tilde h\in H$ and $v,\tilde v\in V$. We first study the two simpler cases: ${\rm Ric}(h,v)$ and ${\rm Ric}(v,\tilde v)$. It is easy to see that ${\rm Ric}(h,v)$ vanishes, because
\begin{align*}
R(h_i,h)v&=-\frac14[h_i,\mu(v,h)]+\frac14[h,\mu(v,h_i)]\in V,\\
R(v_\alpha,h)v&=\frac14\mu(v_\alpha,\mu(v,h))\in H,
\end{align*}
and thus $\langle R(h_i,h)v,h_i\rangle=\langle R(v_\alpha,h)v,v_\alpha\rangle=0$. This explains the zeros in the antidiagonal of~\eqref{ricci_tensor}. To compute ${\rm Ric}(v,\tilde v)$, we need to recall that $R(v_\alpha,v)\tilde v=0$, from Lemma~\ref{curv_end}, and then
\begin{align*}
{\rm Ric}(v,\tilde v)&=\sum_{i=1}^n\langle R(h_i,v)\tilde v,h_i\rangle+\sum_{\alpha=1}^m\langle \underbrace{R(v_\alpha,v)\tilde v}_{=0},v_\alpha\rangle\\
&=-\frac14\sum_{i=1}^n\langle\mu(v,\mu(\tilde v,h_i)),h_i\rangle=\frac14\sum_{i=1}^n\langle\mu(\tilde v,h_i),\mu(v,h_i)\rangle\\
&=\frac14\sum_{i=1}^n\langle v,\tilde v\rangle\langle h_i,h_i\rangle=\left(\frac14\sum_{i=1}^n\varepsilon_i^{\nu_H}\right)\langle v,\tilde v\rangle,
\end{align*}
which, in the basis $(v_1,\ldots, v_m)$, corresponds to the lower diagonal block in~\eqref{ricci_tensor}. Finally, we need to compute ${\rm Ric}(h,\tilde h)$, and we do it by pieces. First, we have
\begin{align}\label{1st_part_Rich}
\sum_{\alpha=1}^m\langle R(v_\alpha,h)\tilde h,v_\alpha\rangle&=\frac14\sum_{\alpha=1}^m\langle[h,\mu(v_\alpha,\tilde h)],v_\alpha\rangle=\frac14\sum_{\alpha=1}^m\langle\mu(v_\alpha,h),\mu(v_\alpha,\tilde h)\rangle\nonumber\\
&=\frac14\sum_{\alpha=1}^m\langle v_\alpha,v_\alpha\rangle\langle h,\tilde h\rangle=\left(\frac14\sum_{\alpha=1}^m\varepsilon_\alpha^{\nu_V}\right)\langle h,\tilde h\rangle.
\end{align}
On the other hand, we see that since $\langle \mu([\tilde h,h_i], h),h_i\rangle=0$
\begin{align*}
\sum_{i=1}^n\langle R(h_i,h)\tilde h,h_i\rangle&=\frac12\sum_{i=1}^n\langle \mu([h_i,h],\tilde h),h_i\rangle-\frac14\sum_{i=1}^n\langle \mu([\tilde h,h_i], h),h_i\rangle\\
&=\frac12\sum_{i=1}^n\langle[h_i,h],[\tilde h,h_i]\rangle-\frac14\sum_{i=1}^n\langle[\tilde h,h_i],[h,h_i]\rangle\\
&=-\frac34\sum_{i=1}^n\langle[h,h_i],[\tilde h,h_i]\rangle.
\end{align*}
In order to compute the last expression, we make use of the Fourier decomposition $[h,h_i]=\sum_{\alpha=1}^m\langle[h,h_i],v_\alpha\rangle v_\alpha$ and obtain
\begin{align}\label{2nd_part_Rich}
\sum_{i=1}^n\langle R(h_i,h)\tilde h,h_i\rangle&=-\frac34\sum_{i=1}^n\left\langle\sum_{\alpha=1}^m\langle[h,h_i],v_\alpha\rangle v_\alpha,[\tilde h,h_i]\right\rangle\nonumber\\
&=-\frac34\sum_{\alpha=1}^m\left\langle\left[\tilde h,\sum_{i=1}^n\langle[h,h_i],v_\alpha\rangle h_i\right],v_\alpha\right\rangle\nonumber\\
&=-\frac34\sum_{\alpha=1}^m\left\langle\left[\tilde h,\sum_{i=1}^n\langle\mu(v_\alpha,h),h_i,\rangle h_i\right],v_\alpha\right\rangle\nonumber\\
&=-\frac34\sum_{\alpha=1}^m\langle[\tilde h,\mu(v_\alpha,h)],v_\alpha\rangle=\left(-\frac34\sum_{\alpha=1}^m\varepsilon_\alpha^{\nu_V}\right)\langle h,\tilde h\rangle,
\end{align}
where the last equality follows the lines of~\eqref{1st_part_Rich}. Combining~\eqref{1st_part_Rich} and~\eqref{2nd_part_Rich}, we have
\[
{\rm Ric}(h,\tilde h)=\left(-\frac12\sum_{\alpha=1}^m\varepsilon_\alpha^{\nu_V}\right)\langle h,\tilde h\rangle,
\]
which, in the basis $(h_1,\ldots,h_n)$, corresponds to the upper diagonal block in~\eqref{ricci_tensor}.
\end{proof}

A natural consequence of this theorem is a closed expression of the scalar curvature $S$ of ${\mathbb G}$ in terms of the sign symbols $\varepsilon_i^\nu$.

\begin{corollary}
In the notations of Theorem~\ref{ricci_theorem}, we have that
\[
S=\tr({\rm Ric})=-\frac14\left(\sum_{\alpha=1}^m\varepsilon_\alpha^{\nu_V}\right)\left(\sum_{i=1}^n\varepsilon_i^{\nu_H}\right).
\]
In particular we have the values
\begin{itemize}
\item $S=-\frac14$ for ${\mathscr H}^{2n,n,1}$, \quad$S=-\frac{n}2$ for ${\mathscr H}^{2n,0,1}$, and
\item $S=-\frac14$ for ${\mathscr H}^{4n,2n,3}$, \quad $S=-3n$ for ${\mathscr H}^{4n,0,3}$.
\end{itemize}
\end{corollary}

\begin{remark}
Observe that when constructing the examples ${\mathscr H}^{2n,n,1}$ and ${\mathscr H}^{4n,2n,3}$, we made a choice of a vector $u_0$. Changing that choice could change the signs of their scalar curvatures. This seems to be a new phenomenon for general $H$-type groups, since for left-invariant positive definite metrics on nilpotent groups it is known, see~\cite{Jensen}, that the scalar curvature is negative.
\end{remark}


\begin{thebibliography}{99}

\bibitem{ABB} {\sc A.~Agrachev, D.~Barilary, U.~Boscain,} Introduction to Riemannian and sub-Riemannian geometry. Manuscript in preparation, available on the web-site: http://www.cmapx.polytechnique.fr/~barilari/Notes.php

\bibitem{BR} 
{\sc A.~Bella\"{\i}che, J.~J.~Risler}, Sub-Riemannian geometry, Progress in Mathematics, {\bf 144}, Birkh\"auser, Basel, 1996.

\bibitem{CCM} {\sc O. Calin, D.-C. Chang, I. Markina}, {\it Geometric analysis on $H$-type groups related to division algebras}. Math.
Nachr. {\bf 282} (2009), no. 1, pp. 44--68. 

\bibitem{CDPT}
{\sc L.~Capogna, D.~Danielli, S.~D.~Pauls, J.~T.~Tyson,} {An introduction to the Heisenberg group and the sub-Riemannian isoperimetric problem.} Progress in Mathematics, {\bf 259}. Birkh\"auser Verlag, Basel, 2007.

\bibitem{CM} {\sc D.-C. Chang, I. Markina}, {\it Geometric analysis on quaternion ${\mathbb H}$-type groups}. J. Geom. Anal. {\bf 16} (2006), no. 2, pp. 265--294.

\bibitem{CMV}
{\sc D.~C.~Chang, I.~Markina, A.~Vasil'ev}, {\it Sub-Lorentzian geometry on anti-de Sitter space}, J. Math. Pures Appl. (9) {\bf 90} (2008), no.~1, pp. 82--110.

\bibitem{C} {\sc W. L.Chow}, {\it \"Uber Systeme von linearen partiellen Differentialgleichungen erster Ordnung}. Math. Ann. {\bf 117} (1939), pp. 98--105.

\bibitem{CDKR} {\sc M.~Cowling, A.~H.~Dooley, A.~Kor\'anyi, F.~Ricci}, {\it $H$-type groups and Iwasawa decompositions.} Adv. Math. 87 (1991), no. 1, pp. 1--41.

\bibitem{FS} {\sc G. B. Folland, E. M. Stein}, Hardy spaces on homogeneous groups, Mathematical Notes, 28, Princeton Univ. Press, Princeton, NJ, 1982.

\bibitem{G} {\sc M. Grochowski}, {\it Reachable sets for the Heisenberg sub-Lorentzian structure on ${\mathbb R}^3$. An estimate for the distance function}. J. Dyn. Control Syst. {\bf 12} (2006), no. 2, pp. 145--160.

\bibitem{G1} {\sc M. Grochowski}, {\it 
Properties of reachable sets in the sub-Lorentzian geometry.} J. Geom. Phys. {\bf 59} (2009), no. 7, pp. 885--900.

\bibitem{G2} {\sc M. Grochowski}, {\it 
Normal forms and reachable sets for analytic Martinet sub-Lorentzian structures of Hamiltonian type.} J. Dyn. Control Syst. {\bf 17} (2011), no. 1, pp. 49--75.

\bibitem{GV}
{\sc E.~Grong, A.~Vasil'ev,}
{\it Sub-Riemannian and sub-Lorentzian geometry on $SU(1,1)$ and on its universal cover.}
J. Geom. Mech. {\bf 3} (2011), no. 2, pp. 225--260. 

\bibitem{Jensen}, {\sc G. R. Jensen}, {\it The scalar curvature of left-invariant Riemannian metrics}. Indiana Univ. Math. J. {\bf 20} (1970/1971), pp. 1125--1144.

\bibitem{K1} {\sc A. Kaplan}, {\it Fundamental solutions for a class of hypoelliptic PDE generated by composition of quadratics forms}. Trans. Amer. Math. Soc. {\bf 258} (1980), no. 1, pp. 147--153.

\bibitem{K2} {\sc A. Kaplan}, {\it Riemannian nilmanifolds attached to Clifford modules}. Geom. Dedicata {\bf 11} (1981), no. 2, pp. 127--136.

\bibitem{K3} {\sc A. Kaplan}, {\it On the geometry of groups of Heisenberg type}. Bull. London Math. Soc., {\bf 15} (1983), no. 1, pp. 35--42.

\bibitem{KM1} {\sc A. Korolko, I. Markina}, {\it Nonholonomic Lorentzian geometry on some H-type groups}. J. Geom. Anal. {\bf 19} (2009), no. 4, pp. 864--889.

\bibitem{KM2} {\sc A. Korolko, I. Markina}, {\it Geodesics on H-type quaternion groups with sub-Lorentzian metric and their physical interpretation}. Complex Anal. Oper. Theory {\bf 4} (2010), no. 3, pp. 589--618.

\bibitem{KM3} {\sc A. Korolko, I. Markina}, {\it Semi-Riemannian geometry with nonholonomic constraints}. Taiwanese J. Math. {\bf 15} (2011), no. 4, pp. 1581--1616.

\bibitem{Lam} {\sc T. Y. Lam}, The algebraic theory of quadratic forms. Mathematics Lecture Note Series. W. A. Benjamin, Inc., Reading, Mass., 1973.

\bibitem{Lee} {\sc J. M. Lee}, Introduction to topological manifolds, second edition, Graduate Texts in Mathematics, {\bf 202}, Springer, New York, 2011. 

\bibitem{Milnor} {\sc J. Milnor}, {\it Curvatures of left invariant metrics on Lie groups}, Adv. Math. {\bf 21} (1976), no.~3, pp. 293--329.

\bibitem{M} {\sc R. Montgomery}, A tour of subriemannian geometries, their geodesics and applications. Mathematical Surveys and Monographs, 91. American Mathematical Society, Providence, RI, 2002.

\bibitem{ON} {\sc B. O'Neill}, Semi-Riemannian geometry. With applications to relativity. Pure and Applied Mathematics, 103. Academic Press, Inc. 

\bibitem{R} {\sc P. K. Rashevski{\u\i}}, {\it About connecting two points of complete nonholonomic space by admissible curve}, Uch. Zapiski Ped. Inst. K.~Liebknecht {\bf 2} (1938), pp. 83--94.

\bibitem{S}
{\sc R.~S.~Strichartz},  {\it Sub-Riemannian geometry},  J. Differential
Geom. {\bf 24} (1986) pp. 221--263; Correction, ibid. {\bf 30}
(1989) pp. 595-596.

\end{thebibliography}
\end{document}